\documentclass[11pt,twoside, a4paper,reqno]{amsart}
\pdfoutput=1

\usepackage[shorthands=off,english]{babel}
\usepackage[T1]{fontenc}
\usepackage[utf8]{inputenc}
\usepackage{amsmath,amsfonts,amssymb}
\usepackage{graphicx}
\usepackage{stmaryrd}
\usepackage{float}
\usepackage{enumitem}
\usepackage{caption}
\usepackage{tikz}
\usetikzlibrary{patterns}
\usetikzlibrary{decorations.markings} 
\usetikzlibrary{decorations.pathreplacing}
\usepackage{tikz-cd}
\usepackage[all]{xy}
\usepackage[]{color}
\usepackage{xcolor}
\usepackage{textcomp}
\usepackage{latexsym}
\usepackage{xspace}
\usepackage{mathrsfs}
\usepackage[babel=false]{csquotes}
\MakeOuterQuote{"}
\usepackage{chngcntr}
\usepackage{etoolbox}
\DeclareMathAlphabet{\mathcal}{OMS}{cmsy}{m}{n}
\usepackage{longtable}
\definecolor{mymauve}{rgb}{0.4,0,1}
\usepackage[hidelinks,hypertexnames=false]{hyperref}
\hypersetup{
  colorlinks   = true,  
  urlcolor     = blue,  
  linkcolor    = blue,  
  citecolor   = mymauve 
}
\usepackage{amsthm}
\usepackage[english]{cleveref-forward}

\usepackage[margin=1in]{geometry}
\theoremstyle{plain}
\newtheorem{theorem}{Theorem}[section]{}{}
\newtheorem*{theorem*}{Theorem}
\newtheorem{maintheorem}{Theorem}

\newtheorem{lemma}[theorem]{Lemma}
\newtheorem{corollary}[theorem]{Corollary}
\newtheorem{proposition}[theorem]{Proposition}

\theoremstyle{definition}

\newtheorem{definition}[theorem]{Definition}
\theoremstyle{remark}
\newtheorem{remark}[theorem]{Remark}
\newtheoremstyle{TheoremNum}
    {8pt}{8pt}                            
    {\itshape}                            
    {}                                    
    {\bfseries}                           
    {.}                                   
    { 5pt plus 1pt minus 1pt }            
    {\thmname{#1}\thmnote{ \bfseries #3}} 
\theoremstyle{TheoremNum}

\makeatletter
\def\cref@thmoptarg[#1]#2#3#4{%
    \ifhmode\unskip\unskip\par\fi%
    \normalfont%
    \trivlist%
    \let\thmheadnl\relax%
    \let\thm@swap\@gobble%
    \thm@notefont{\fontseries\mddefault\upshape}%
    \thm@headpunct{.}
    \thm@headsep 5\p@ plus\p@ minus\p@\relax%
    \thm@space@setup%
    #2
    \@topsep \thm@preskip               
    \@topsepadd \thm@postskip           
    \def\@tempa{#3}\ifx\@empty\@tempa%
      \def\@tempa{\@oparg{\@begintheorem{#4}{}}[]}%
    \else%
      \refstepcounter[#1]{#3}
      \@namedef{cref@#3@alias}{#1}
      \def\@tempa{\@oparg{\@begintheorem{#4}{\csname the#3\endcsname}}[]}%
    \fi%
    \@tempa}%
\makeatother

\crefname{corollary}{Corollary}{Corollaries}
\crefname{conjecture}{Conjecture}{Conjectures}
\crefname{theorem}{Theorem}{Theorems}
\crefname{maintheorem}{Theorem}{Theorems}
\crefname{proposition}{Proposition}{Propositions}
\crefname{lemma}{Lemma}{Lemmas}
\crefname{example}{Example}{Examples}
\crefname{definition}{Definition}{Definitions}
\crefname{remark}{Remark}{Remarks}
\crefname{figure}{Figure}{Figures}
\crefname{table}{Table}{Tables}
\crefname{section}{Section}{Sections}
\crefname{appendix}{Appendix}{Appendices}

\def\co{\colon}
\newcommand{\tq}{\mathrel{{\ensuremath{\: : \: }}}}
\def\wt{\widetilde}
\let\emptyset\varnothing
\def\actson{\curvearrowright}
\def\rightAction{\curvearrowleft}
\def\groupIso{\simeq}
\def\llangle{\langle\!\langle}
\def\rrangle{\rangle\!\rangle}
\newcommand\BigFreeProd{\displaystyle\mathop{\mbox{\huge{$\ast$}}}}
\def\unsplitExtension{\cdot}

\newcommand{\Aut}{\mathrm{Aut}}
\newcommand{\centralizer}{{\mathbf{C}}}
\newcommand{\centerOfGroup}{{\mathbf{Z}}}

\def\GL{\mathrm{GL}}
\def\PSL{\mathrm{PSL}}
\def\SL{\mathrm{SL}}
\def\Sz{\mathrm{Sz}}
\def\lieGroup{\mathbb{G}}
\def\U{\mathbf{U}}

\newcommand{\rk}{\mathrm{rk}}
\newcommand{\diag}{\mathrm{diag}}
\def\sg{\mathrm{sg}}
\def\ii{\mathbf{i}}
\def\Res{\mathrm{Res}}
\def\Ind{\mathrm{Ind}}
\def\Ad{\mathrm{Ad}}

\def\C{\mathbb{C}}
\def\F{\mathbb{F}}
\def\N{\mathbb{N}}
\def\R{\mathbb{R}}
\def\Z{\mathbb{Z}}

\def\FF{\mathcal{F}}
\def\K{\mathcal{K}}

\def\SLV{\mathcal{SLV}}
\def\XOS{X_1^{OS}}
\def\XOSk{X_1^{OS+k}}
\def\oneComplex{X_1}
\def\inclusionBrown{i}
\def\inclusionStabilizerVertex{i}

\def\MM{\mathcal{M}}
\def\MMbar{\overline{\mathcal{M}}}
\def\MMbarbar{\overline{\overline{\mathcal{M}}}}
\def\HH{\mathcal{H}}

\def\rhobar{\overline{\rho}}
\def\rhoG{{\rho_0}}

\def\Wbar{\overline{W}}
\def\XX{\mathbf{X}}
\def\WW{\mathbf{W}}

\def\WWbar{\overline{\mathbf{W}}}
\def\YYbar{\overline{\mathbf{Y}}}

\def\identityLieGroup{\mathbf{1}}
\def\identityMM{\identityLieGroup}
\def\identityMMBar{\overline{\identityMM}}
\def\taubar{{\overline{\tau}}}

\def\onlyEdgeInEMinusT{{\tilde{\eta}}}

\title{Group actions on contractible $2$-complexes I}
\author[I. Sadofschi Costa]{Iv\'an Sadofschi Costa}

\address{Departamento  de Matem\'atica - IMAS\\
 FCEyN, Universidad de Buenos Aires. Buenos Aires, Argentina.}
\email{isadofschi@dm.uba.ar}



\thanks{Researcher of CONICET. The author was partially supported by grants PICT-2017-2806, PIP 11220170100357CO and UBACyT 20020160100081BA}

\begin{document}
  \begin{abstract}
  In this series of two articles, we prove that every action of a finite group $G$ on a finite and contractible $2$-complex has a fixed point.
  The proof goes by constructing a nontrivial representation of the fundamental group of each of the acyclic $2$-dimensional $G$-complexes constructed by Oliver and Segev.
  In the first part we develop the necessary theory and cover the cases where $G=\PSL_2(2^n)$, $G=\PSL_2(q)$ with $q\equiv 3\pmod 8$  or $G=\Sz(2^n)$.
  The cases $G=\PSL_2(q)$ with $q\equiv 5\pmod 8$ are addressed in the second part.
  \end{abstract}

\maketitle

\setcounter{tocdepth}{1}

\vspace{-0.15in}

\tableofcontents

\vspace{-0.35in}

\section{Introduction}
A well known result of Jean-Pierre Serre states that an action of a finite group on a tree has a fixed point~\cite{SerreTrees}.
A natural attempt to generalize Serre's result would be to replace ``tree'' by ``contractible $n$-complex''.
An example by Edwin E. Floyd and Roger W. Richardson~\cite{FloydRichardson} implies this generalization does not hold for $n\geq 3$.
However, Carles Casacuberta and Warren Dicks conjectured that it holds for $n=2$~\cite{CD}.
In the compact case and in the form of a question, this was also posed by Michael Aschbacher and Yoav Segev~\cite[Question 3]{AschbacherSegev}.
In this series of two articles, we give a positive answer to the question of Aschbacher--Segev, settling the compact case of the Casacuberta--Dicks conjecture.

\begin{maintheorem}\label{CasacubertaDicks}
 Every action of a finite group $G$ on a $2$-dimensional finite and contractible complex has a fixed point.
 Moreover, if $G$ is a finite group and $X$ is a $2$-dimensional, fixed point free, finite and acyclic $G$-complex, then the fundamental group of $X$ admits a nontrivial unitary representation.
\end{maintheorem}

In~\cite{CD} the conjecture is proved for solvable groups.
The question of which groups act without fixed points on a finite acyclic $2$-complex was studied independently by Segev~\cite{Segev}, who proved this is not possible for the solvable groups and the alternating groups $A_n$ for $n\geq 6$.
Using the classification of the finite simple groups, Aschbacher and Segev proved that for many groups any action on a finite $2$-dimensional acyclic complex has a fixed point~\cite{AschbacherSegev}.

Then, Bob Oliver and Yoav Segev~\cite{OS} gave the complete classification of the groups that act without fixed points on an acyclic $2$-complex.
A concise introduction to this subject was given by Alejandro Adem at the S\'eminaire Bourbaki~\cite{Adem}.
The main results in  Oliver and Segev's classification are the following (see \cref{DefEssential} for the definition of an essential $G$-complex).

\begin{theorem}[{Oliver--Segev}]\label{OSThmA}
 For any finite group $G$, there is an essential fixed point free $2$-dimensional (finite) acyclic $G$-complex if and only if $G$ is isomorphic to one of the simple groups $\PSL_2(2^k)$ for $k\geq 2$, $\PSL_2(q)$ for $q\equiv \pm 3 \pmod 8$ and $q\geq 5$, or $\Sz(2^k)$ for odd $k\geq 3$. Furthermore, the isotropy subgroups of any such $G$-complex are all solvable.
\end{theorem}

\begin{theorem}[{Oliver--Segev}]\label{OSThmB}
Let $G$ be any finite group, and let $X$ be any $2$-dimensional acyclic $G$-complex. Let $N$ be the subgroup generated by all normal subgroups $N'\triangleleft G$ such that $X^{N'}\neq \emptyset$. Then $X^N$ is acyclic; $X$ is essential if and only if $N=1$; and the action of $G/N$ on $X^N$ is essential.
\end{theorem}

In an unpublished preprint~\cite{A5contractible}, the author proved the $G=A_5\groupIso \PSL_2(2^2)$ case of \cref{CasacubertaDicks} and proposed a path to prove \cref{CasacubertaDicks},
which consists of representing (in a nontrivial way) the fundamental group of each of the acyclic $2$-complexes constructed by Oliver and Segev.
Since this reduction is needed to prove \cref{CasacubertaDicks} and the preprint~\cite{A5contractible} will remain unpublished, we reproduce the argument in \cref{SectionReduction}.
Combining the Oliver--Segev classification with the Gerstenhaber--Rothaus theorem, we deduce \cref{CasacubertaDicks} from \cref{thmB,thmC} below.

\begin{maintheorem}\label{thmB}
 Let $G$ be one of the groups $\PSL_2(2^n)$ for $n\geq 2$, $\PSL_2(3^n)$ for $n\geq 3$ odd, $\PSL_2(q)$ with $q\equiv 11\pmod {24}$ or $q\equiv 19 \pmod{24}$, or $\Sz(q)$ for $q=2^n$ with $n\geq 3$ odd.
 Then the fundamental group of every $2$-dimensional, fixed point free, finite and acyclic $G$-complex admits a nontrivial representation in a unitary group $\U(m)$.
\end{maintheorem}

\begin{maintheorem}[{\cite{part2}}]\label{thmC}
 Let $G$ be one of the groups $\PSL_2(q)$ with $q>5$ and $q\equiv 5\pmod {24}$ or $q\equiv 13 \pmod{24}$.
 Then the fundamental group of every $2$-dimensional, fixed point free, finite and acyclic $G$-complex admits a nontrivial representation in a unitary group $\U(m)$.
\end{maintheorem}
The proof of \cref{thmC} appears in the second part of this work~\cite{part2}, which is joint with Kevin Piterman.

To prove \cref{thmB,thmC}, we use the method of~\cite{A5contractible} but with a more generic approach.
If $\oneComplex$ is a $G$-graph we consider the group extension $\Gamma = \pi_1(\oneComplex,x_0)\unsplitExtension G$. 
If $X$ is obtained from $\oneComplex$ by attaching orbits of $2$-cells, a result of Kenneth S. Brown~\cite{BrownPresentations} gives an extension
$\Gamma/\llangle w_0,\ldots, w_k\rrangle\groupIso \pi_1(X)\unsplitExtension G$,
where the $w_i\in \ker(\phi\colon \Gamma\to G)\groupIso \pi_1(\oneComplex)$
are words corresponding to the orbits of $2$-cells of $X$.
Then obtaining a nontrivial representation of $\pi_1(X)$ reduces to obtaining a representation of $\Gamma$ which factors through the quotient $\Gamma\to \Gamma/\llangle w_0,\ldots, w_k\rrangle$ and does not factor through $\phi$.

In this paper we develop general machinery to obtain a moduli of representations $\MMbar$ of $\Gamma$ from a single representation $\rhoG\colon G\to \lieGroup$, where $\lieGroup$ is a Lie group.
Each word $w\in\Gamma$ induces a map $W\colon \MMbar\to \lieGroup$ and then the proof reduces to finding a suitable point $\taubar\in \MMbar$.
With some hypotheses on $\rhoG\colon G\to \lieGroup$, there is a single point $\identityMMBar\in \MMbar$ which gives a representation that factors through $\phi$.
Then, by considering $\WWbar=(W_0,\ldots,W_k)\colon \MMbar \to \lieGroup^{k+1}$, the proof reduces to finding a point $\taubar\neq \identityMMBar \in \MMbar$ such that $\WWbar(\taubar)=\identityLieGroup$.
When we apply these results to the groups in \cref{OSThmA} it turns out that $\MMbar$ and $\lieGroup^{k+1}$ are orientable manifolds of the same dimension.
To complete the proof we show that $\identityMMBar$ is a regular point of $\WWbar$ and that $\WWbar$ has degree $0$.

The groups in \cref{thmB} share a key property: they admit a nontrivial representation which restricts to an irreducible representation of the Borel subgroup.
However, the groups in \cref{thmC} lack this property.
In~\cite{part2} some modifications to the approach of the first part are introduced in order to extend the proof to these groups.

\bigskip

\textbf{Acknowledgements.}
I am grateful to Kevin Piterman for his help in understanding the structure and representations of the finite simple groups $\PSL_2(q)$ and $\Sz(q)$ and for many fruitful discussions.
I am also grateful to Ignacio Darago for answering all of my questions on representation theory and Lie theory.
I would like to thank Jonathan Barmak and Gabriel Minian for useful comments on an earlier version of this work.
Finally, I am grateful to the anonymous referee for the detailed review and valuable suggestions which simplified some proofs improving the exposition.


\section{The results of Oliver and Segev}\label{SectionOliverSegev}

In this section we recall the results from~\cite{OS} that are needed later.
By \textit{$G$-complex} we always mean a $G$-CW complex.
That is, a CW complex with a continuous $G$-action that is \textit{admissible}
(i.e. the action permutes the open cells of $X$, and maps a cell to itself only via the identity).
For more details, see~\cite[Appendix A]{OS}.
A \textit{graph} is a $1$-dimensional CW complex. 
By \textit{$G$-graph} we mean a $1$-dimensional $G$-complex.

\begin{definition}[{\cite{OS}}]\label{DefEssential}
A $G$-complex $X$ is \textit{essential} if there is no normal subgroup $1\neq N\triangleleft G$ such that for each $H\subseteq G$, the inclusion $X^{HN}\to X^H$ induces an isomorphism on integral homology.
\end{definition}

The following fundamental result of Segev~\cite[Theorem 3.4]{Segev} will be used frequently, sometimes implicitly.
We state the version given in~\cite{OS}.

\begin{theorem}[{\cite[Theorem 4.1]{OS}}]\label{AciclicoOVacio}
Let $X$ be any $2$-dimensional acyclic $G$-complex (not necessarily finite). Then $X^G$ is acyclic or empty, and is acyclic if $G$ is solvable.
\end{theorem}

\begin{definition}[{\cite{OS}}]\label{defSeparating}
By a \textit{family} of subgroups of $G$ we mean any set $\FF$ of subgroups of $G$ which is closed under conjugation.
A nonempty family is said to be \textit{separating} if it has the following three properties: (a) $G\notin \FF$; (b) if $H'\subseteq H$ and $H\in \FF$ then $H'\in \FF$; (c) for any $H\triangleleft K\subseteq G$ with $K/H$ solvable, $K\in \FF$ if $H\in \FF$.

For any family $\FF$ of subgroups of $G$, a \textit{$(G,\FF)$-complex} is a $G$-complex all of whose isotropy subgroups lie in $\FF$.
A $(G,\FF)$-complex is \textit{universal} (resp. H-\textit{universal}) if the fixed point set of each $H\in \FF$ is contractible (resp. acyclic).
\end{definition}

If $G$ is not solvable, the separating family of solvable subgroups of $G$ is denoted by $\SLV$.

\begin{lemma}[{\cite[Lemma 1.2]{OS}}]\label{lemma1.2}
Let $X$ be any $2$-dimensional acyclic $G$-complex without fixed points.
Let $\FF$ be the set of subgroups $H\subseteq G$ such that $X^H\neq \emptyset$.
Then $\FF$ is a separating family of subgroups of $G$, and $X$ is an {\normalfont H}-universal $(G,\FF)$-complex.
\end{lemma}

\begin{proposition}[{\cite[Proposition 6.4]{OS}}]\label{proposition6.4}
 Assume that $L$ is one of the simple groups $\PSL_2(q)$ or $\Sz(q)$, where $q=p^k$ and $p$ is prime ($p=2$ in the second case).
 Let $G\subseteq\Aut(L)$ be any subgroup containing $L$, and let $\FF$ be a separating family for $G$.
 Then there is a $2$-dimensional acyclic $(G,\FF)$-complex if and only if $G=L$, $\FF=\SLV$, and $q$ is a power of $2$ or $q\equiv \pm 3 \pmod 8$. 
\end{proposition}

If $X$ is a poset, then  $\K(X)$ denotes the \textit{order complex} of $X$, that is, the simplicial complex with simplices the finite nonempty totally ordered subsets of $X$ (the complex $\K(X)$ is also known as the \textit{nerve} of $X$).

\begin{definition}[{\cite[Definition 2.1]{OS}}]
For any family $\FF$ of subgroups of $G$ define 
\[i_\FF(H)=\frac{1}{[N_G(H):H]}(1-\chi(\K(\FF_{>H}))).\]
\end{definition}

Recall that if $G\actson X$, the orbit $G\cdot x$ is said to be \textit{of type $G/H$} if the stabilizer $G_x$ is conjugate to $H$ in $G$.
In other words, if the action of $G$ on $G\cdot x$ is the same as the action of $G$ on $G/H$.

\begin{lemma}[{\cite[Lemma 2.3]{OS}}]\label{lemma2.3}
Fix a separating family $\FF$, a finite {\normalfont H}-universal $(G,\FF)$-complex $X$, and a subgroup $H\subseteq G$.
For each $n$, let $c_n(H)$ denote the number of orbits of $n$-cells of type $G/H$ in $X$. Then $i_\FF(H)=\sum_{n\geq 0} (-1)^nc_n(H)$.
\end{lemma}

\begin{proposition}[{\cite[Tables 2,3,4]{OS}}]\label{indices}
Let $G$ be one of the simple groups $\PSL_2(2^k)$ for $k\geq 2$, $\PSL_2(q)$ for $q\equiv \pm 3 \pmod 8$ and $q\geq 5$, or $\Sz(2^k)$ for odd $k\geq 3$. Then $i_\SLV( 1 ) = 1$.
\end{proposition}

For each family of groups appearing in \cref{OSThmA}, Oliver and Segev describe an example.
In what follows, $D_{2m}$ is a dihedral group of order $2m$ and $C_m$ is a cyclic group of order $m$.

\begin{proposition}[{\cite[Example 3.4]{OS}}]\label{example3.4}
Set $G=\PSL_2(q)$, where $q=2^k$ and $k\geq 2$.
Then there is a $2$-dimensional acyclic fixed point free $G$-complex $X$, all of whose isotropy subgroups are solvable.
More precisely $X$ can be constructed to have three orbits of vertices with isotropy subgroups isomorphic to $B=\F_q\rtimes C_{q-1}$, $D_{2(q-1)}$, and $D_{2(q+1)}$; three orbits of edges with isotropy subgroups isomorphic to $C_{q-1}$, $C_2$ and $C_2$; and one free orbit of $2$-cells. 
\end{proposition}

\begin{proposition}[{\cite[Example 3.5]{OS}}]\label{example3.5}
Assume that $G=\PSL_2(q)$, where $q=p^k\geq 5$ and $q\equiv \pm 3 \pmod 8$. Then there is a $2$-dimensional acyclic fixed point free $G$-complex $X$, all of whose isotropy subgroups are solvable. More precisely, $X$ can be constructed to have four orbits of vertices with isotropy subgroups isomorphic to $B=\F_q\rtimes C_{(q-1)/2}$, $D_{q-1}$, $D_{q+1}$, and $A_4$; four orbits of edges with isotropy subgroups isomorphic to $C_{(q-1)/2}$, $C_2^2$, $C_3$ and $C_2$; and one free orbit of $2$-cells.
\end{proposition}

\begin{proposition}[{\cite[Example 3.7]{OS}}]\label{example3.7}
Set $q=2^{2k+1}$ for any $k\geq 1$. Then there is a $2$-dimensional acyclic fixed point free $\Sz(q)$-complex $X$, all of whose isotropy subgroups are solvable. More precisely, $X$ can be constructed to have four orbits of vertices with isotropy subgroups isomorphic to $M(q,\theta)$, $D_{2(q-1)}$, $C_{q+\sqrt{2q}+1}\rtimes C_4$, $C_{q-\sqrt{2q}+1}\rtimes C_4$; four orbits of edges with isotropy subgroups isomorphic to $C_{q-1}$, $C_4$, $C_4$ and $C_2$; and one free orbit of $2$-cells.
\end{proposition}

In all three cases, the stabilizers for the orbits of vertices are precisely the maximal solvable subgroups of $G$ (this is key to construct the examples, see \cite[Section 3]{OS} for more details).
Note also that $\PSL_2(4)\cong\PSL_2(5)\cong A_5$, so this group is addressed in both \cref{example3.4} and \cref{example3.5}.
There is no other such exception.


\section{A reduction}
\label{SectionReduction}

In this section we build upon the results of Oliver and Segev to prove \cref{Refinement},
which reduces the proof of \cref{thmB,thmC} to the special case given by the acyclic $2$-complexes of the type constructed in~\cite{OS}.

\begin{definition}
If $G$ is one of the groups in \cref{OSThmA}, the \textit{Oliver--Segev graph $\XOS(G)$} is
the $1$-skeleton of any $2$-dimensional fixed point free acyclic $G$-complex
of the type constructed in \cref{example3.4,example3.5,example3.7}.
For any $k\geq 0$, we also consider the $G$-graph $\XOSk(G)$
obtained from $\XOS(G)$ by attaching $k$ free orbits of $1$-cells.
\end{definition}

For the previous definition we regard $A_5$ as $\PSL_2(2^2)$ rather than $\PSL_2(5)$.
Generally, there is more than one possible choice for the $G$-graph $\XOS(G)$.
Even for $G=A_5$, thought of as $\PSL_2(2^2)$, the quotient graph $\XOS(G)/G$ is not unique.

\begin{definition}
If $X,Y$ are $G$-spaces, a \textit{$G$-homotopy} is an equivariant map $H\co X\times I \to Y$.
We say that $f_0(x)=H(x,0)$ and $f_1(x)=H(x,1)$ are \textit{$G$-homotopic} and we denote this by $f_0\simeq_G f_1$.
An equivariant map $f\co X\to Y$ is a \textit{$G$-homotopy equivalence} if there is an equivariant map $g\co Y\to X$ such that $fg\simeq_G 1_Y$ and $gf\simeq_G 1_X$.
A $G$-invariant subspace $A$ of $X$ is a \textit{strong $G$-deformation retract of $X$} if there is a retraction $r\co X\to A$ such that there is a $G$-homotopy $H\co ir\simeq 1_X$ relative to $A$, where $i\co A\to X$ is the inclusion.
\end{definition}

\begin{remark}
An equivariant map $f\co X\to Y$ is a $G$-homotopy equivalence if and only if $f^H\co X^H \to Y^H$ is a homotopy equivalence for each subgroup $H\leq G$ (see~\cite[Chapter II, (2.7) Proposition]{Dieck}).
Thus, if $f\co X\to Y$ is a $G$-homotopy equivalence, the action $G\actson X$ is fixed point free (resp. essential) if and only if the action $G\actson Y$ is fixed point free (resp. essential).
\end{remark}

The following explains why our choice of $\XOS(G)$
and the way the free orbits of $1$-cells are attached to it to obtain $\XOSk(G)$ is not relevant for our purposes.

\begin{proposition}\label{UnicidadXOS}
The graph $\XOS(G)$ is unique up to $G$-homotopy equivalence.
Moreover, for any $k\geq 0$, $\XOSk(G)$ is unique up to $G$-homotopy equivalence.
\begin{proof}
Since any choice of $\XOS(G)$ is a universal $(G,\mathcal{SLV}-\{1\})$-complex, the first part follows from~\cite[Proposition A.6]{OS}.
The second part follows easily from the first,
the fact that homotopic attaching maps give rise to homotopy equivalent adjunction spaces {\cite[7.5.5 (Corollary 1)]{TopologyGroupoids}},
and the gluing theorem for adjunction spaces {\cite[7.5.7]{TopologyGroupoids}}.
\end{proof}
\end{proposition}

\begin{corollary}\label{PosiblesPi1NoDependenDeGammaOS}
The set of $G$-homotopy equivalence classes of $2$-dimensional acyclic $G$-complexes
which can be obtained from $\XOSk(G)$ by attaching $k+1$ free orbits of $2$-cells
does not depend on the choice of $\XOSk(G)$.
In particular, the set of isomorphism classes of groups
that occur as the fundamental group of such spaces does not depend on such choices.
\begin{proof}
Again, this is an easy application of {\cite[7.5.5 (Corollary 1)]{TopologyGroupoids}} and {\cite[7.5.7]{TopologyGroupoids}}.
\end{proof}
\end{corollary}

The following lemma allows us to do elementary expansions equivariantly.
\begin{lemma}\label{LemaExpansiones}
Let $X$ be an acyclic $2$-dimensional $G$-complex. Let $H\leq G$ and $x_0,x_1\in X^{(0)}\cap X^H$.
Then there is a $G$-complex $Y\supset X$ that strong $G$-deformation retracts to $X$
and is obtained from $X$ by attaching an orbit of $1$-cells of type $G/H$ with endpoints $\{ x_0 , x_1\}$ and an orbit of $2$-cells of type $G/H$.
 
\begin{proof}
We attach an orbit of $1$-cells of type $G/H$ to $X$ using the attaching map $\varphi\colon G/H \times S^0 \to X^{(0)}$
defined by  $\varphi(gH,1)= g\cdot x_0$, $\varphi(gH,-1)= g\cdot x_1$.
Let $e$ be the $1$-cell of this new orbit corresponding to the coset $H$. Since $X$ is acyclic, by \cref{AciclicoOVacio} $X^H$ is also acyclic.
Let $\gamma$ be an edge path in $X^H$ starting at $x_1$ and ending at $x_0$.
Then we attach an orbit of $2$-cells of type $G/H$ in such a way
that the $2$-cell corresponding to the coset $H$ is attached along the closed edge path given by concatenating $e$ and $\gamma$.
It is clear that $X$ is a strong $G$-deformation retract of $Y$.
\end{proof}
\end{lemma}

We recall some very natural definitions which appear in~\cite[Section 2]{KLV}.
A \textit{forest} is a graph with trivial first homology.
If a subcomplex $\Gamma$ of a CW complex $X$ is a forest, there is a CW complex $Y$ obtained from $X$ by shrinking each connected component of $\Gamma$ to a point.
The quotient map $q\co X\to Y$ is a homotopy equivalence and we say $Y$ is obtained from $X$ by a \textit{forest collapse}.
If $X$ is a $G$-complex and $\Gamma\subset X$ is a forest which is $G$-invariant,
the quotient map $q$ is a $G$-homotopy equivalence
and we say the $G$-complex $Y$ is obtained from $X$ by a $G$-\textit{forest collapse}.
We say that a $G$-graph is \textit{reduced} if it has no edge $e$ such that $G\cdot e$ is a forest.

\begin{lemma}\label{Lema1EsqueletoReducido}
Let $X$ be a $2$-dimensional acyclic $G$-complex. If $X^{(1)}$ is a reduced $G$-graph then stabilizers of different vertices are not comparable.

\begin{proof}
Let $\FF=\{ G_x \tq x\in X^{(0)}\}$ and let $M=\{ v\in X^{(0)} \tq G_v \text{ is maximal in }\FF \}$.
We first prove, by contradiction, that $X^{(0)}=M$.
Consider $v\in X^{(0)}-M$ such that $G_v$ is maximal in $\{G_{x} \tq x \in X^{(0)}-M\}$.
Then since $X^{G_v}$ contains $v$, by \cref{AciclicoOVacio} it must be acyclic.
Since $v\notin M$, there is a vertex $w\in X^{G_v}\cap M$.
By connectivity there is an edge $e\in X^{G_v}$ whose endpoints $v'$ and $w'$ satisfy $v'\notin M$ and $w'\in M$.
Since $G_{v'}\geq G_v$ and $v'\notin M$, by our choice of $v$ we have $G_v=G_{v'}$.
Since $e\in X^{G_v}$ we have $G_v\leq G_e$ and since $v'$ is an endpoint of $e$ we have $G_e \leq G_{v'}$. 
Thus $G_e=G_{v'}$ and then the degree of $v'$ in the graph $G\cdot e$ (whose vertex set is the disjoint union of $G\cdot w'$ and $G\cdot v'$) is $1$.
Thus $G\cdot e$ is a forest, contradiction.
Therefore we must have $M=X^{(0)}$.

To conclude we have to prove that different vertices $u,v\in M$ have different stabilizers.
Suppose $G_u=G_v$ to get a contradiction.
Since $u,v$ are vertices of $X^{G_u}$ which is connected, there is an edge $e\in X^{G_u}$ and by maximality we must have $G_e=G_u$.
If $u',v'$ are the endpoints of $e$, we have $G_{u'}=G_{v'}$. We have two cases and in any case we obtain a contradiction.
If $G\cdot u' \neq G \cdot v'$ then $G\cdot e$ is a forest consisting of $|G/G_e|$ disjoint edges, contradiction.
Otherwise, there is a nontrivial element $g\in G$ such that $g\cdot u' = v'$ and we have $G_{u'}= G_{v'} = gG_{u'}g^{-1}$. Thus $g\in N_G( G_{u'} )$.
Consider the action of $\langle g \rangle$ on $X^{G_{u'}}$, which is acyclic and thus has a fixed point by the Lefschetz fixed point theorem.
But this cannot happen, since it would imply that $\langle G_{u'}, g \rangle \gneq G_{u'}$ fixes a point of $X$, which is a contradiction because $u'\in M$.
\end{proof}
\end{lemma}

Now we prove the main result of this section.

\begin{theorem}\label{Refinement}
 Let $G$ be one of the groups in \cref{OSThmA}.
 Let $X$ be a fixed point free $2$-dimensional finite acyclic $G$-complex.
 Then there is a fixed point free $2$-dimensional finite acyclic $G$-complex $X'$
 obtained from $\XOSk(G)$ (for some $k\geq 0$) by attaching
 $k+1$ free orbits of $2$-cells and an epimorphism $\pi_1(X)\to \pi_1(X')$.
  \begin{proof}
    Let $\FF=\{H\leq G \tq X^H\neq \emptyset \}$.
    Then, by \cref{lemma1.2}, $\FF$ is a separating family and $X$ is an {\normalfont{H}}-universal $(G,\FF)$-complex.
    By \cref{proposition6.4}, we must have $\FF=\mathcal{SLV}$.
    By doing enough $G$-forest collapses we can assume that $X^{(1)}$ is a reduced $G$-graph.
    The stabilizers of the vertices of $\XOS(G)$ are precisely the maximal solvable subgroups of $G$.
    Therefore, since every solvable subgroup of $G$ fixes a point of $X$,
    by \cref{Lema1EsqueletoReducido}, we may identify  $X^{(0)}=\XOS{(G)}^{(0)}$.
    Applying \cref{LemaExpansiones} enough times to modify $X$, we may further assume $\XOS(G)$ is a subcomplex of $X$.

    Finally we will modify $X$ so that for every subgroup $1\neq H\leq G$, we have $X^H=\XOS{(G)}^H$.
    We do this by reverse induction on $|H|$.
    Assume that we have $X$ such that our claim holds for every subgroup $K$ with $H< K\leq G$.
    If $H$ is not solvable, we have $X^H=\XOS(G)^H=\emptyset$ so we are done.
    If $H$ is solvable, since $\XOS{(G)}^H$ is a tree (it is acyclic and $1$-dimensional) and $X^H$ is acyclic by \cref{AciclicoOVacio},
    the inclusion $\XOS{(G)}^H\hookrightarrow X^H$ is an $N_G(H)$-equivariant homology equivalence. 
    Now, since $\XOS{(G)}^H$ is a tree, we can define an $N_G(H)$-equivariant retraction $r_H\co  X^H\to \XOS{(G)}^H$.
    Then $r_H$ is a homology equivalence.
    Moreover, the stabilizer of the cells in $X^H - \XOS(G)^H$ is $H$ (the stabilizer cannot be bigger by the induction hypothesis).
    We define retractions $r_{H^g}\co X^{H^g}\to \XOS{(G)}^{H^g}$ by  $r_{H^g}(gx)=g\cdot r_H(x)$ which glue to give a $G$-equivariant homology equivalence
      \[r\co \XOS(G) \bigcup_{g\in G} X^{H^g}\to \XOS(G).\]
    We may replace $X$ by the pushout $\wt{X}$ given by the following diagram
    \begin{center}
    \begin{tikzcd}
    \XOS(G) \displaystyle\bigcup_{g\in G} X^{H^g} \arrow[hook]{d}[]{}\arrow{r}[]{r} & \XOS(G)\arrow[hook]{d}{} \\
    X\arrow[]{r}[swap]{\overline{r}}& \wt{X}
    \end{tikzcd}
    \end{center}
    It follows that $\overline{r}$ is a homology equivalence, so the resulting $G$-complex $\wt{X}$ is acyclic.
    Moreover since $\wt{X}^{(1)}$ is a subcomplex of $X^{(1)}$ and the restriction $\overline{r}\colon X^{(1)}\to \wt{X}^{(1)}$ is a retraction,
    $\overline{r}$ induces an epimorphism on $\pi_1$.
    This procedure removes the excessive orbits of cells of type $G/H$.
    By induction we obtain a complex $X'$ such that $X'^{(1)}$ coincides with $\XOS(G)$ up to $k\geq 0$ free orbits of $1$-cells and such that every orbit of $2$-cells is free.
    By \cref{lemma1.2} $X'$ is an H-universal $(G,\mathcal{SLV})$-complex.
    Now by \cref{lemma2.3} and \cref{indices} there are exactly $k+1$ orbits of $2$-cells.
\end{proof}
\end{theorem}

We conclude this section by describing, for each of the groups $G$ in \cref{OSThmA},
a feasible way to connect the orbits in the graph $\XOS(G)$.
The following lemma will be handy.

\begin{lemma}\label{lemmaSlidingOneCells}
  Let $G$ be a finite group and let $X_1$ be a $G$-graph.
  Let $u,v,w$ be vertices of $X_1$ and let $e,e'$ be edges such that $e$ has endpoints $\{u,v\}$ and $e'$ has endpoints $\{v,w\}$.
  Suppose that $G_{e'}\subseteq G_e$.
  Consider the $G$-graph $Y_1$ obtained from $X_1$ by removing the orbit of $e'$ and attaching an orbit $e''$ of edges of type $G/G_{e'}$ with endpoints $\{u,w\}$
  (i.e. the attaching map $\varphi\colon G/G_{e'} \times S^0 \to X_1 - G\cdot e'$ for the orbit of $e''$ is defined by $\varphi(gH,1)=gu$, $\varphi(gH,-1)=gw$).
  Then $X_1$ and $Y_1$ are $G$-homotopy equivalent.
  \begin{proof}
  The graphs $X_1$ and $Y_1$ are $G$-homotopy equivalent because both are obtained from $X_1 - G \cdot e'$ by attaching an orbit of $1$-cells of type $G/G_{e'}$
  and these attaching maps are $G$-homotopic (the homotopy can be easily written down using $e$ and its orbit).
  \end{proof}
 \end{lemma}
 
 \begin{figure}[ht]
   \begin{minipage}{0.48\textwidth}
       \centering
       \begin{tikzcd}
       B \arrow[-]{d}[swap]{C_{q-1}} \arrow[-]{dr}[]{C_2} & \\
       D_{2(q-1)} & D_{2(q+1)}\arrow[-]{l}[]{C_2}
       \end{tikzcd}
       
        \bigskip
       
       $G=\PSL_2(2^n)$.
   \end{minipage}
   \begin{minipage}{0.48\textwidth}
       \centering
       \begin{tikzcd}
       B \arrow[-]{d}[swap]{C_{\frac{q-1}{2}}}\arrow[-]{r}[]{C_3} & A_4\arrow[-]{d}[]{C_2^2} \\
       D_{q-1} \arrow[-]{r}[swap]{C_2} & D_{q+1}
       \end{tikzcd}
 
       \bigskip
       
       $G=\PSL_2(3^n)$, odd $n$.
   \end{minipage}
   \bigskip
   
   \begin{minipage}{0.48\textwidth}
     \centering
     \begin{tikzcd}
     B \arrow[-]{d}[swap]{C_{\frac{q-1}{2}}}\arrow[-]{r}[]{C_3} & A_4\arrow[-]{d}[]{C_2^2} \\
     D_{q-1} \arrow[-]{r}[swap]{C_2} & D_{q+1}
     \end{tikzcd}
     
       \bigskip
       
     $G=\PSL_2(q)$, $q\equiv 19\pmod {24}$.
    \end{minipage}
    \begin{minipage}{0.48\textwidth}
     \centering
     \begin{tikzcd}
     B \arrow[-]{d}[swap]{C_{\frac{q-1}{2}}} & A_4\arrow[-,bend left=20]{d}[]{C_2^2} \arrow[-,bend right=20]{d}[swap]{C_3} \\
     D_{q-1} \arrow[-]{r}[swap]{C_2} & D_{q+1}
     \end{tikzcd}
     
       \bigskip
       
       $G=\PSL_2(q)$, $q\equiv 11\pmod {24}$.
    \end{minipage}
    
    \bigskip
    
    \begin{minipage}{0.48\textwidth}
       \centering
       \begin{tikzcd}
       B \arrow[-]{d}[swap]{C_{\frac{q-1}{2}}} \arrow[-]{r}[]{C_3}& A_4\arrow[-]{dl}[]{C_2^2} \\
       D_{q-1} \arrow[-]{r}[swap]{C_2} & D_{q+1}
       \end{tikzcd}
       
       \bigskip
       
       $G=\PSL_2(q)$, $q\equiv 13\pmod {24}$.
    \end{minipage}
   \begin{minipage}{0.48\textwidth}
     \centering
     \begin{tikzcd}
     B \arrow[-]{d}[swap]{C_{\frac{q-1}{2}}}\arrow[-]{r}[]{C_2}  & D_{q+1} \arrow[-]{d}[]{C_3} \\
     D_{q-1} \arrow[-]{r}[swap]{C_2^2} & A_4
     \end{tikzcd}
     
     \bigskip
     
     $G=\PSL_2(q)$, $q\equiv 5\pmod {24}$.
    \end{minipage}
    
    \bigskip
    
      \begin{minipage}{0.48\textwidth}
       \centering
       \begin{tikzcd}
       B \arrow[-]{r}[]{C_4}  \arrow[-]{d}[swap]{C_{q-1}}&  C_{q-r+1}\rtimes C_4\arrow[-]{d}[]{C_4} \\
       D_{2(q-1)} \arrow[-]{r}[swap]{C_2} & C_{q+r+1}\rtimes C_4
       \end{tikzcd}
       
       \bigskip
       
       $G=\Sz(q)$, $q=2^n$.
    \end{minipage}
    \caption{
      One of the possible ways to construct the orbits of $\XOS(G)$ in each case.
      The figure depicts, in each case, the quotient graph $\XOS(G)/G$.
      The label for each vertex (resp. edge) is the isomorphism class of the stabilizer $G_v$ (resp. $G_e$) of a representative in $\XOS(G)$ of the vertex (resp. edge).}
    \label{FigureXOSG}
 \end{figure}

 \begin{proposition}\label{HowOrbitsConnectInXOS}
 For each of the groups $G$ in \cref{OSThmA}, we can construct $\XOS(G)$ as in \cref{FigureXOSG}.
 \begin{proof}
 In all cases, the orbit types must be those given by \cref{example3.4,example3.5,example3.7}.
 If $H\leq G$ is cyclic of order $k\leq 4$ then, by \cref{remarkOneConjugacyClassSmallCyclicSubgroups},
 $\XOS{(G)}^H$ must intersect every orbit of cells of type $G/K$, provided that $K$ contains a subgroup isomorphic to $H$.
 Moreover by \cref{AciclicoOVacio}, the graph $\XOS{(G)}^H$ is a tree.
 This imposes some restrictions on how orbits are connected and at the same time gives us freedom to mutate $\XOS(G)$ by applying \cref{lemmaSlidingOneCells}.

 For the groups $G=\PSL_2(q)$, a possible way to connect the orbits is described in~\cite[Section 3]{OS}.
 For each of these groups, the structure in \cref{FigureXOSG} coincides with this one up to an application of \cref{lemmaSlidingOneCells}.
 
 For $G=\Sz(q)$ we give more detail here.
 Let $r=\sqrt{2q}$.
 First note that, since $q-1\nmid 4(q\pm r+1)$, the orbit of type $C_{q-1}$ has to connect $B$ to $D_{2(q-1)}$.
 Now the two orbits of type $C_4$ must connect $B$, $C_{q+r+1}\rtimes C_4$ and $C_{q-r+1}\rtimes C_4$ (in some way).
 The orbit $C_2$ must connect $D_{2(q-1)}$ to one of the other three orbits of vertices.
 Note that, in any case, we can repeatedly use \cref{lemmaSlidingOneCells} to obtain the desired structure.
 \end{proof}
 \end{proposition}


\section{Preliminaries on Lie groups}
\label{SectionLieGroups}

Recall that a \textit{Lie group} $G$ is a smooth manifold with a group structure such that the multiplication $\mu\colon G\times G\to G$, $(x,y)\mapsto xy$ and inversion $i\colon G\to G$, $x\mapsto x^{-1}$ are differentiable.
The group $\U(m)$ of $m\times m$ unitary matrices is a compact and connected $m^2$-dimensional Lie group.
If $G$ is a Lie group, the \textit{Lie algebra} of $G$ is the tangent space $T_e G$ at the identity element $e\in G$.
The \textit{adjoint representation} $\Ad\colon G \to \GL(T_e G)$ is defined by $g\mapsto d_e\Psi_g$ where $\Psi_g\colon G\to G$ is the map given by $h\mapsto ghg^{-1}$.
Every Lie group is parallelizable and hence orientable.
 
 \begin{lemma}\label{differentialMu}
 Let $G$ be a Lie group with multiplication $\mu\colon G\times G\to G$. Then the differential $d_{(p,q)}\mu\colon T_p G\times T_q G \to T_{pq} G$ is given by $(x,y)\mapsto d_p R_q (x) + d_q L_{p} (y)$.
 \begin{proof}
 The differential $d_{(e,e)}\mu \colon T_e G \times T_e G \to T_e G$ is given by $(x,y)\mapsto x+y$ (this is~\cite[Chapter 7, Problem 7-2]{Lee}).
 The general case follows by writing $\mu = L_p R_q \circ \mu \circ (L_{p^{-1}} \times R_{q^{-1}})$.
 \end{proof}

\end{lemma}

\begin{proposition}\label{propProductRule}
  Let $M$ be a manifold, $G$ be a Lie group and $f,g\colon M\to G$ be differentiable maps.
  \begin{enumerate}[label=(\roman*)]
  \item We have the product rule $d_p(f\cdot g) = d_{f(p)} R_{g(p)} \circ d_p f + d_{g(p)} L_{f(p)} \circ  d_p g$.
  
  \item If $f(p)=g(p)=e$, we have  $d_p(f\cdot g) = d_p f + d_p g$.
  
  \item If $g(p)=e$, we have  $d_p (f\cdot g \cdot f^{-1}) = d_{f(p)^{-1}} L_{f(p)} \circ d_{e} R_{f(p)^{-1}} \circ  d_p g $.
  
  \item If $f(p)=e$, we have $d_p f^{-1}=-d_p f$.
  
  \item If $f(p)=g(p)=e$, we have $d_p [f,g] = 0$.
  \end{enumerate}
  \begin{proof}
  These properties follow easily from \cref{differentialMu}.
  \end{proof}
\end{proposition}

\begin{corollary}\label{DescriptionAd}
 The adjoint representation is given by $\Ad(g)=d L_g \circ d R_{g^{-1}}$.
\end{corollary}

We denote the centralizer of $H$ in $G$ by $\centralizer_G(H)$ and the center of $G$ by $\centerOfGroup(G)$.

\begin{proposition}[{\cite[Chapter III, \S 9, no. 3, Proposition 8]{BourbakiLie}}]\label{LieAlgebraOfCentralizer}
 Let $H$ be a finite subgroup of a Lie group $G$.
 Then the Lie algebra of the centralizer $\centralizer_G( H )$ is obtained by taking the fixed points by $H$ of the adjoint representation of $G$. That is, we have $T_e \centralizer_G( H ) = (T_e G)^H$.
\end{proposition}

\begin{theorem}\label{centralizersAreConnectedInUnitaryGroup}
 Let $H\leq \U(m)$ be a subgroup. Then $\centralizer_{\U(m)}(H)$ is connected.
 \begin{proof}
 A proof using a simultaneous diagonalization argument is given in~\cite[Proof of Theorem 3.2]{Stacey}. See also~\cite{CentralizersInUnitaryGroup}.
 \end{proof}
\end{theorem}

\begin{proposition}[{\cite[Corollary 21.6]{Lee}}]\label{ActionOfCompactLieGroupIsProper}
Every continuous action by a compact Lie group on a manifold is proper.
\end{proposition}

\begin{theorem}[Quotient Manifold Theorem]\label{QuotientManifoldTheorem}
Suppose $G$ is a Lie group acting smoothly, freely, and properly on a smooth manifold $M$.
Then the orbit space $M/G$ is a topological manifold of dimension equal to $\dim M - \dim G$, and has a unique smooth structure with the property that the quotient map $\pi\colon M\to M/G$ is a smooth submersion.

Moreover, if $M$ is orientable and $G$ is connected, then $M/G$ is orientable.
\begin{proof}
 The first part is~\cite[Theorem 21.10]{Lee}.
 For the second part we fix an orientation on $M$ and $G$.
 Since $G$ is connected, the translations $L_g,R_g\colon G\to G$ and $g\colon M\to M$ are homotopic to the identity map and thus preserve the orientation.
 A tedious but straightforward computation with the charts constructed in the proof of~\cite[Theorem 21.10]{Lee} allows to extract an oriented atlas, showing that $M/G$ is orientable.
\end{proof}
\end{theorem}


\section{A moduli of representations of $\Gamma = \pi_1(\oneComplex,x_0)\unsplitExtension G$}\label{sectionModuli}

If $\oneComplex$ is a connected $G$-graph, there is a group extension
\[1\to \pi_1(\oneComplex,v_0) \xrightarrow{\inclusionBrown} \Gamma \xrightarrow{\phi} G\to 1\]
which is most easily defined by lifting the action of $G$ to the universal cover $\wt{\oneComplex}$ of $\oneComplex$.
In this section we construct a moduli $\MM$ of representations of the group extension $\Gamma$ and study its properties
(note that we are using the word \textit{moduli} in a rather informal way, meaning a geometric object whose points correspond to certain representations of $\Gamma$).
The starting point to construct $\MM$ is a result in Bass--Serre theory due to K.S.~Brown which provides (at the price of making some choices)
a much more concrete description of $\Gamma$ that allows us to work with it~\cite{BrownPresentations}.
\begin{theorem}[Brown]
\label{Brown}
 Let $X$ be obtained from a $G$-graph $\oneComplex$ by attaching $m$ orbits of $2$-cells along (the orbits of) the closed edge paths $\omega_0,\ldots,\omega_k$ based at a vertex $v_0$.
 Then there is a group extension
 \[1\to \pi_1(X,v_0) \xrightarrow{\overline{\inclusionBrown}} \Gamma / \llangle \inclusionBrown({\omega_0}), \ldots, \inclusionBrown({\omega_k})\rrangle \xrightarrow{\overline{\phi}} G\to 1,\]
 where the maps $\overline{\inclusionBrown}$ and $\overline{\phi}$ are given by factoring through the quotient.
\end{theorem}
In order to describe Brown's construction of $\Gamma$ and the maps $\inclusionBrown$ and $\phi$ we need some choices.
By admissibility of the action, the group $G$ acts on the set of oriented edges. 
If $e$ is an oriented edge, the same $1$-cell with the opposite orientation is denoted by $e^{-1}$.
Each oriented edge $e$ has a \textit{source} $s(e)$ and a \textit{target} $t(e)$.
For each $1$-cell of $\oneComplex$ we choose a preferred orientation in such a way that these orientations are preserved by $G$.
This determines a set $P$ of oriented edges.
We choose a \textit{tree of representatives} for $\oneComplex/G$.
That is, a tree $T\subset \oneComplex$ such that the vertex set $V$ of $T$ is a set of representatives of $\oneComplex^{(0)}/G$.
Such tree always exists and the $1$-cells of $T$ are inequivalent modulo $G$.
We give an orientation to the $1$-cells of $T$ so that they belong to $P$.
We also choose a set of representatives $E$ of $P/G$ in such a way that $s(e)\in V$ for every $e\in E$ and such that each oriented edge of $T$ is in $E$.
If $e$ is an oriented edge, the unique element of $V$ that is equivalent to $t(e)$ modulo $G$ will be denoted by $w(e)$.
For every $e\in E$ we fix an element $g_e\in G$ such that $t(e)=g_e\cdot w(e)$.
If $e\in T$, we specifically choose $g_e=1$.
Then
    \[\Gamma = \frac{F(x_e\tq e\in E ) * \BigFreeProd_{v\in V} G_v}{\llangle R \rrangle },\]
where $F(x_e\tq e\in E )$ is the free group with basis $\{x_e\tq e\in E\}$ and $\llangle R \rrangle$ denotes the normal subgroup generated by the set $R$ of relations of the following two types:
\begin{enumerate}[label=(\roman*)]
\item $x_e=1$ if $e\in T$, and

\item $x_e^{-1}\iota_{s(e)}(g) x_e = \iota_{w(e)}(g_e^{-1}gg_e)$ for every $e\in E$ and $g\in G_e$,
\end{enumerate}
where $\iota_v\colon G_v\hookrightarrow F(x_e\tq e\in E ) * \BigFreeProd_{v\in V} G_v$ denotes the canonical inclusion.

Let $\phi\colon \Gamma\to G$ be the map induced by the coproduct of the inclusions $G_v\to G$ and the map $F(x_e\tq e\in E )\to G$ defined by $x_e\mapsto g_e$.
Let $N=\ker(\phi)=i(\pi_1(\oneComplex,v_0))$.
Let $\inclusionStabilizerVertex_v\colon G_v\hookrightarrow \Gamma$ be the canonical inclusion.
We will not give the description of $i$ here, instead we refer to~\cite{BrownPresentations} or to~\cite[Section 4]{A5contractible}.

In the following proposition we use a morphism $\rhoG\colon G\to \lieGroup$ to construct a moduli of representations of $\Gamma$ in the Lie group $\lieGroup$.
This extends the construction in~\cite[Theorem 5.4]{A5contractible}.

\begin{theorem}\label{constructionModuli}
Let $\oneComplex$ be a $G$-graph with the necessary choices to apply \cref{Brown}.
Take a vertex $v_0\in V$ as the root of $T$ and assume the orientation $P$ is taken
so that every edge of $T$ is oriented away from $v_0$.
Let $\Gamma$ be the group given by Brown's result and consider a representation
$\rhoG\colon G\to \lieGroup$ of $G$ in a Lie group $\lieGroup$.
Let $\MM = \displaystyle\prod_{e\in E} \centralizer_{\lieGroup}( \rhoG(G_e) )$.
Suppose $\tau=(\tau_e)_{e\in E}\in \MM$. For $v\in V$, we define
$\tau_v=\tau_{e_k}\tau_{e_{k-1}}\cdots \tau_{e_2}\tau_{e_1}$
where $(e_1,e_2,\ldots,e_k)$ is the unique path from $v_0$ to $v$ by edges in $T$
(with this definition $\tau_{v_0}=\identityLieGroup$).
Then we have a representation $\rho_\tau\colon \Gamma \to \lieGroup$ given by

\begin{align*}
 \rho_\tau( \inclusionStabilizerVertex_v( g ) ) &= \tau_v^{-1} \rhoG(g) \tau_v & \text{ for } v\in V \text{ and } g\in G_v,\\
 \rho_\tau( x_e ) &= \tau_{s(e)}^{-1} \tau_e^{-1} \rhoG( g_e )\tau_{w(e)} & \text{ for } e\in E.
\end{align*}
We thus have a moduli of representations
\begin{align*}
\rho\colon \MM &\to \hom( \Gamma, \lieGroup) \\
\tau&\mapsto \rho_\tau
\end{align*}
Moreover, each word $w\in \Gamma$ induces a differentiable map $W\colon \MM\to \lieGroup$ given by $\tau\mapsto \rho_\tau(w)$.
\begin{proof}
If $e\in T$ then $\tau_{w(e)}=\tau_{t(e)} = \tau_e\tau_{s(e)}$ and $g_e=1$. Therefore $\rho_\tau(x_e)=\identityLieGroup$ and relations of type (i) are satisfied.
Now if $e\in E$, $g\in G_e$ we have
\begin{align*}
\rho_\tau(x_e)^{-1} \rho_\tau( \inclusionStabilizerVertex_{s(e)}(g) )\rho_\tau( x_e ) &= \tau_{w(e)}^{-1}\rhoG(g_e)^{-1} \tau_e \tau_{s(e)} \cdot  \tau_{s(e)}^{-1} \rhoG( g) \tau_{s(e)} \cdot \tau_{s(e)}^{-1}\tau_e^{-1} \rhoG(g_e) \tau_{w(e)}\\
&= \tau_{w(e)}^{-1} \rhoG{(g_e)}^{-1}  \tau_e \rhoG( g ) \tau_e^{-1}\rhoG(g_e)  \tau_{w(e)}\\
&= \tau_{w(e)}^{-1} \rhoG{(g_e)}^{-1} \rhoG(g) \rhoG(g_e) \tau_{w(e)}\\
&= \rho_\tau( \inclusionStabilizerVertex_{w(e)}(g_e^{-1} g g_e) )
\end{align*}
and thus the type (ii) relations
$x_e^{-1} \inclusionStabilizerVertex_{s(e)}( g ) x_{e} = \inclusionStabilizerVertex_{w(e)}(g_e^{-1} g g_e )$
also hold.

Finally, proving that for $w\in \Gamma$ the map $W\colon \MM \to \lieGroup$ is differentiable
reduces to proving that the maps $\tau\mapsto \rho_\tau(i_v(g))$ and $\tau\mapsto \rho_\tau(x_e)$ are differentiable.
This in turn follows from $\tau \mapsto \tau_e$ and $\tau\mapsto \tau_v$ being differentiable.
\end{proof}
\end{theorem}

Different points of $\MM$ may correspond to equal representations of $\Gamma$.
The quotient $\MMbar$ introduced in the following result allows us to deal with this issue.

\begin{theorem}\label{quotientModuli}
Let $\HH = \{(\alpha_v)_{v\in V} \tq \alpha_{v_0}=\identityLieGroup \} \subseteq \displaystyle\prod_{v\in V } \centralizer_\lieGroup(\rhoG(G_v))$.
Assume $\HH$ is compact.
\begin{enumerate}[label=(\roman*)]
\item There is a free right action $\MM\rightAction \HH$ given by
 \[{(\tau \cdot \alpha)}_e = \rhoG(g_e) \alpha_{w(e)}^{-1} \rhoG(g_e)^{-1}   \cdot \tau_e \cdot \alpha_{s(e)}\]
 
\item Moreover $\rho_{\tau} = \rho_{\tau'}$ if and only if $\tau,\tau'$ lie in the same orbit of the action of $\HH$.
 
\item The quotient $\MMbar = \MM/\HH$ is a smooth manifold, the map $p\colon \MM\to\MMbar$ is a smooth submersion and $\dim \MMbar = \dim \MM - \dim \HH$.
 
\item If $\HH$ is connected then $\MMbar$ is orientable.

\item We have an induced map $\rhobar\colon \MMbar \to \hom(\Gamma, \lieGroup)$. Each word $w\in \Gamma$ induces a differentiable map $\overline{W}\colon \MMbar \to \lieGroup$ such that $\overline{W} = W \circ p$.
\end{enumerate}

\begin{proof}
(i) Since $G_{s(e)}\supseteq G_e\subseteq G_{t(e)}$, the fact that the given action is well-defined follows from $ \rhoG(g_e) \alpha_{w(e)}^{-1}\rhoG(g_e)^{-1}\in \centralizer_\lieGroup( \rhoG( G_{t(e)} ))$ which holds since $t(e)= g_e \cdot w(e)$.
 If $(\tau \cdot \alpha)_e  = \tau_e$ for all $e\in T$, by induction (traversing the tree $T$ starting from the root $v_0$) it follows that $\alpha_v=1$ for all $v\in V$.
 Then the action is free.
 
 (ii) Let $\tau\in \MM$, $\alpha\in \HH$. If $e\in T$ then $(\tau\cdot \alpha)_e = \alpha_{t(e)}^{-1} \tau_e \alpha_{s(e)}$.
 If $v\in V$, $(\tau\cdot \alpha)_{v} = \alpha_v^{-1}\tau_v$.
 Then
 \begin{align*}
  \rho_{\tau\cdot \alpha}(\inclusionStabilizerVertex_v(g)) &= (\tau\cdot \alpha)_v^{-1} \rhoG(g)(\tau\cdot \alpha)_v\\
			    &= \tau_v^{-1}\alpha_v\rhoG(g)\alpha_v^{-1}\tau_v \\
			    &= \tau_v^{-1}\rhoG(g)\tau_v \\
			    &= \rho_\tau(\inclusionStabilizerVertex_v(g)).
 \end{align*}
 Moreover, for $e\in E$ we have
\begin{align*}
 \rho_{\tau\cdot \alpha}( x_e ) &= (\tau\cdot \alpha)_{s(e)}^{-1} (\tau\cdot \alpha)_e^{-1} \rhoG( g_e )(\tau\cdot \alpha)_{w(e)} \\
&= (\alpha_{s(e)}^{-1}\tau_{s(e)})^{-1} ( \rhoG(g_e) \alpha_{w(e)}^{-1} \rhoG(g_e)^{-1} \tau_e \alpha_{s(e)})^{-1}\rhoG(g_e)  (\alpha_{w(e)}^{-1}\tau_{w(e)})\\
 &=\tau_{s(e)}^{-1} \tau_e^{-1} \rhoG( g_e )\tau_{w(e)}\\
 &= \rho_{\tau}(x_e).
 \end{align*}
 Then $\rho_\tau = \rho_{\tau\cdot \alpha}$.
 For the other implication, if $\tau,\tau'\in\MM$ satisfy $\rho_\tau = \rho_{\tau'}$, by defining $\alpha_v = \tau_v (\tau'_v)^{-1}$ we obtain a point $\alpha = (\alpha_v)_{v\in V}\in \HH$ and $\tau\cdot \alpha=\tau'$.

 (iii) By \cref{ActionOfCompactLieGroupIsProper}  the action is proper. Then by \cref{QuotientManifoldTheorem},
 the quotient  $\MMbar = \MM/\HH$ has a (unique) smooth manifold structure such that $p\colon \MM\to \MMbar$ is a submersion
 and $\dim \MMbar = \dim \MM - \dim \HH$.
 
 (iv) This follows from the second part of \cref{QuotientManifoldTheorem}.
 
 (v) This follows by passing to the quotient.
\end{proof}
\end{theorem}

\begin{corollary}\label{connectednessOfMAndMbar}
If $\lieGroup=\U(m)$ then $\MM$ and $\MMbar$ are connected and orientable.
\begin{proof}
 In view of \cref{centralizersAreConnectedInUnitaryGroup}, it follows from part (iv) of \cref{quotientModuli}.
\end{proof}
\end{corollary}

A representation $\rho\colon \Gamma \to \lieGroup$ is said to be \textit{universal} if $N\subseteq \ker(\rho)$ (or equivalently, if $\rho$ factors through $\phi$).
Under suitable hypotheses, $\identityMMBar=p(\identityMM)$ is the only point in $\MMbar$ which corresponds to a universal representation:
\begin{proposition}\label{universalIsUnique}
Suppose that $G$ is finite and that each element of $G$ fixes a vertex in $\oneComplex$.
Let $\lieGroup\subseteq \GL_m(\C)$ and assume the restriction $\rhoG|_{G_{v_0}}\colon G_{v_0}\to \lieGroup$ is an irreducible representation of $G_{v_0}$.
Then $\{ \identityMMBar\}=\{ \taubar \in \MMbar \tq \rhobar_{\taubar} \text{ is universal}\}$.
\begin{proof}
 \newcommand{\rhoQuotient}{\wt{\rho}_\tau}  
  First note that $\rhobar_{\identityMMBar} = \rho_{\identityMM} = \rhoG \circ \phi$ is universal.
 Now consider $\tau\in \MM$ such that $\rho_\tau$ is universal.
 By passing to the quotient we have a representation $\rhoQuotient \colon G\to \lieGroup$ such that
 $\rho_{\tau}=\rhoQuotient \circ \phi$.
 Now note that, since each element of $G$ fixes a vertex of $\oneComplex$, from the definition of $\rho_\tau$ it follows that the representations $\rhoG$ and $\rhoQuotient$ have the same character and are therefore isomorphic.
 Hence, we can take $\alpha\in \GL_m(\C)$ such that for all $g\in G$ we have $\alpha\rhoQuotient(g)\alpha^{-1}=\rhoG(g)$.
 Now since for every $g\in G_{v_0}$ we have $\rhoQuotient(g)=\rhoG(g)$, and since $\rhoG|_{G_{v_0}}$ is irreducible,
 by Schur's lemma it follows that $\alpha$ is a scalar matrix and therefore $\rhoQuotient=\rhoG$.
 Then $\rho_\tau = \rho_\identityMM$ and therefore by part (ii) of \cref{quotientModuli}, $p(\tau)=p(\identityMM)$ in $\MMbar$.
\end{proof}

\end{proposition}

\begin{remark}
 If $\rhoG|_{G_{v_0}}\colon G_{v_0}\to \lieGroup$ is not irreducible, we could still consider the quotient $\MMbarbar$ of $\MMbar$ by the action of $\centralizer_\lieGroup(\rhoG(G_{v_0}))$.
 In this case, the points in $\MMbarbar$ correspond to characters (not representations) of $\Gamma$ and the image of the induced map $\overline{\Wbar}$ is only defined up to conjugation by $\centralizer_\lieGroup(\rhoG(G_{v_0}))$.
 Note that the quotient of $\lieGroup$ by the conjugation action of $\centralizer_\lieGroup(\rhoG(G_{v_0}))$ is not, in general, a manifold.
\end{remark}

The following result relates a closed edge path $\omega\in \oneComplex$ to the differential at $\identityMM$ of the map $\MM\to \lieGroup$ induced by the word $i(\omega)\in \Gamma$.

\begin{theorem}\label{differentialWordInducedFromClosedEdgePath}
 Let $\oneComplex$ be a $G$-graph (with the necessary choices to form $\MM$).
 Consider a closed edge path $\omega = (a_1 e_1^{\varepsilon_1}, \ldots, a_n e_n^{\varepsilon_n})$ in $\oneComplex$, based at $v_0$, with $e_i\in E$, $a_i\in G$ and $\varepsilon_i\in\{1,-1\}$.
 Let $w=\inclusionBrown(\omega) \in N=\ker(\phi)$.
 Let $W\colon \MM\to \lieGroup$ be the induced differentiable map.
 Let $\identityMM =(\identityLieGroup)_{e\in E} \in \MM$ and consider the inclusion $j_e\colon \centralizer_\lieGroup( \rhoG(G_e) ) \hookrightarrow \lieGroup$.
 Then, with the identification $T_\identityMM \MM\simeq \displaystyle\bigoplus_{e\in E} T_{\identityLieGroup} \centralizer_\lieGroup(\rhoG( G_e) )$ we have
    \[\displaystyle d_\identityMM  W = -\sum_{i=1}^n \varepsilon_i\cdot  d_{\rhoG{(a_i)}^{-1}} L_{\rhoG(a_i)}\circ d_\identityMM  R_{\rhoG(a_i)^{-1}} \circ d_\identityMM  j_{e_i}.\]
 
 \begin{proof}
 By the definition of $\inclusionBrown\colon \pi_1(\oneComplex)\to \Gamma$ (see~\cite{BrownPresentations} or~\cite[Section 4]{A5contractible}) we can write
  \[w = \inclusionStabilizerVertex_{v_0}(h_1) \cdot x_{e_1}^{\varepsilon_1} \cdot \inclusionStabilizerVertex_{v_1}(h_2) \cdot x_{e_2}^{\varepsilon_2}\cdots  \inclusionStabilizerVertex_{v_{n-1}}(h_{n}) \cdot x_{e_n}^{\varepsilon_n} \cdot \inclusionStabilizerVertex_{v_0}(g_1g_2\cdots g_{n})^{-1}\]
  so that for each $i$ we have $g_i=h_{i}g_{e_i}^{\varepsilon_i}$ and 
    \[a_i=\begin{cases}
    g_1\cdots g_{i-1}h_i & \text{ if } \varepsilon_i = 1 \\
    g_1\cdots g_{i-1}h_i g_{e_i}^{-1} & \text{ if } \varepsilon_i = -1.
\end{cases}\]
  Then 
  \begin{align*}
W( \tau ) &=\left( \prod_{i=1}^n (\tau_{v_{i-1}}^{-1}\rhoG(h_i)\tau_{v_{i-1}}) (\tau_{s(e_{i})}^{-1} \tau_{e_i}^{-1}\rhoG(g_{e_i}) \tau_{w(e_i)} )^{\varepsilon_i}\right) \tau_{v_0}^{-1}\rhoG(g_1g_2\cdots g_{n})^{-1}\tau_{v_0}  \\
 &=\left( \prod_{i=1}^n \rhoG(h_i) (\tau_{e_i}^{-1}\rhoG(g_{e_i}))^{\varepsilon_i}\right)\rhoG{(g_1g_2\cdots g_{n})}^{-1}.
  \end{align*}
  In the last equality we used that $\tau_{v_0}=1$ and that $s(e_i)$ and $w(e_i)$ are (in some order which depends on $\varepsilon_i$) $v_{i-1}$ and $v_{i}$.
  We have $P_i(\identityMM)=\rhoG(a_i)$ where $P_i$ is the prefix of $W$ ending just before the occurrence of $\tau_{e_i}^{-\varepsilon_i}$.
  Note that, since $W(\identityMM)=\identityLieGroup$, if $S_i$ is the suffix of $W$ starting just after the occurrence of $\tau_{e_i}^{-\varepsilon_i}$, we have $S_i(\identityMM)=\rhoG(a_i)^{-1}$.
  To conclude, we apply the product rule \cref{propProductRule}.
 \end{proof}
\end{theorem}

In what follows $\llangle w_0,\ldots, w_k\rrangle^{\Gamma}$ denotes the normal closure of $\{w_0,\ldots, w_k\}$ in $\Gamma$.

\begin{lemma}[cf. {\cite[Lemma 6.7]{A5contractible}}]\label{theoremDegreeW}
Let $\Gamma$ be a group, $\lieGroup$ be a Lie group, $M$ be a differentiable manifold, and $\rho\colon M\to \hom(\Gamma, \lieGroup)$ be a function such that for each $w\in \Gamma$ the mapping $W\colon M\to \lieGroup$ defined by $W(z)=\rho(z)(w)$ is differentiable.
Let $N\triangleleft \Gamma$ be a normal subgroup and suppose that $p\in M$ is such that $\rho(p)(w)=\identityLieGroup$ for each $w\in N$.
Then for any elements $w_0,\ldots, w_k\in N$ and  $x_0,\ldots, x_k\in \llangle w_0,\ldots, w_k\rrangle^{\Gamma} [N,N]$ we have
$\rk\, d_p \WW\geq \rk\, d_p \XX$, where 
$\WW=(W_0,\ldots,W_k)$ and $\XX=(X_0,\ldots, X_k)$ are the induced maps $M\to \lieGroup^{k+1}$.
\begin{proof}
For each $j=0,\ldots,k$ we consider numbers $a_j, \ell_j\in \N_0$,
elements $u_{j,1},\ldots, u_{j,a_j}, \break v_{j,1},\ldots, v_{j,a_j}\in N$,
elements $p_{j,1},\ldots p_{j,\ell_j}\in \Gamma$,
indices $\alpha_{j,1},\ldots, \alpha_{j,{\ell_j}}\in \{0,\ldots, k\}$
and signs $\varepsilon_{j,1},\ldots,\varepsilon_{j,\ell_j}\in\{1,-1\}$ such that
\[x_j  = \left(\prod_{s=1}^{\ell_j} p_{j,s} w_{\alpha_{j,s}}^{\varepsilon_{j,s}} p_{j,s}^{-1}\right) \prod_{i=1}^{a_j}[u_{j,i},v_{j,i}].\]
Then the induced maps $M\to\lieGroup$ satisfy
\[X_j = \left(\prod_{s=1}^{\ell_j} P_{j,s} W_{\alpha_{j,s}}^{\varepsilon_{j,s}} P_{j,s}^{-1}\right) \prod_{i=1}^{a_j}[U_{j,i},V_{j,i}] \]
and using \cref{propProductRule} we obtain
\[d_{p} X_j = \sum_{s=1}^{\ell_j}  \varepsilon_{j,s} \cdot d_{P_{j,s}{(p)}^{-1}} L_{P_{j,s}(p)} \circ d_e R_{ P_{j,s}{(p)}^{-1} } \circ d_{p} W_{\alpha_{j,s}}.\]
To conclude, note that we have shown there is an $\R$-linear endomorphism $A$ of $T_\identityLieGroup(\lieGroup^{k+1})$ such that
$ d_{p} \XX = A\circ d_{p}\WW$.
\end{proof}

\end{lemma}

We now prove some results that will be used later to obtain homotopies between maps $\MMbar\to \lieGroup$.
We obtain these homotopies from homotopies $\MM\times I \to \lieGroup$ that are $\HH$-equivariant.

\begin{lemma}\label{lemmaRemovingConstantsGeneralization}
 Let $w,w'\in \Gamma$ and let $g\in G_v$ for some $v\in V$.
 Then the maps $\MM \to \lieGroup$ induced by the words $ww'$ and $w i_v(g) w'$ are homotopic.
 Moreover, if $\centralizer_\lieGroup(\centralizer_\lieGroup(\rhoG(G_v)))$ is connected then the same holds for the induced maps $\MMbar\to \lieGroup$.
 \begin{proof}
  Let $W,W'\colon\MM\to\lieGroup$ be the maps induced by $w$ and $w'$ respectively.
  Let $\gamma\colon I\to \lieGroup$ be a path with $\gamma(0)=\identityLieGroup$, $\gamma(1)=\rhoG(g)$.
  The following map
  \begin{align*}
   H\colon \MM \times I &\to \lieGroup\\
   (\tau,t)&\mapsto  W(\tau) \tau_v^{-1} \gamma(t) \tau_v W'(\tau)
  \end{align*}
  is a homotopy between the maps $\MM\to \lieGroup$ induced by $ww'$ and $w i_v(g) w'$.
  Moreover, since $\rhoG(g)\in \centralizer_\lieGroup(\centralizer_\lieGroup(\rhoG(G_v)))$,
  we can take $\gamma(I)\subseteq \centralizer_\lieGroup(\centralizer_\lieGroup(\rhoG(G_v)))$ if the latter is connected and in this case the following computation 
  \begin{align*}
    H(\tau\cdot \alpha,t) &= W(\tau\cdot \alpha) {(\tau\cdot \alpha)}_v^{-1}\gamma(t) {(\tau\cdot \alpha)}_v W'(\tau\cdot \alpha) \\
    &= W(\tau) {(\tau\cdot \alpha)}_v^{-1}\gamma(t) {(\tau\cdot \alpha)}_v W'(\tau) \\
    &= W(\tau) {(\alpha_v^{-1}\tau_v)}^{-1}\gamma(t) (\alpha_v^{-1}\tau_v) W'(\tau) \\
    &= W(\tau) \tau_v^{-1} \alpha_v\gamma(t) \alpha_v^{-1}\tau_v W'(\tau) \\
    &= W(\tau) \tau_v^{-1} \gamma(t) \tau_v W'(\tau) \\
    &= H(\tau,t)
  \end{align*}
  shows $H$ is $\HH$-equivariant, giving a homotopy between the induced maps $\MMbar\to \lieGroup$.
  \end{proof}
\end{lemma}

In the following two propositions we use the notation $\displaystyle\prod_{i=\ell}^1 b_i = b_\ell b_{\ell-1} b_{\ell-2}\cdots b_2 b_1$.

\begin{proposition}\label{homotopyFromPathsInCentralizers}
 Let $\eta\in E-T$ and let $(e_1,\ldots, e_k)$ and $(e'_1,\ldots, e'_\ell)$ be the unique paths in $T$ from $v_0$ to $s(\eta)$ and $w(\eta)$ respectively (see \cref{figureY0}).
 Suppose that $\gamma_0,\ldots,\gamma_k,\beta_0,\ldots,\beta_\ell \colon I\to \lieGroup$ are paths such that:
\begin{itemize}
 \item For $i=1,\ldots,k$ and for every $t\in I$, $\gamma_i(t)$ commutes with $\centralizer_\lieGroup(\rhoG(G_{t(e_i)}))$.
 \item For $i=1,\ldots,\ell$ and for every $t\in I$, $\beta_i(t)$ commutes with $\centralizer_\lieGroup(\rhoG(G_{t(e'_i)}))$.
\end{itemize}
Then there is an $\HH$-equivariant homotopy $F\colon \MM\times I \to \lieGroup$ defined by
\[F(\tau,t) = \gamma_0(t)  \left(\prod_{i=1}^k \tau_{e_i}^{-1} \gamma_i(t) \right)  \tau_{\eta}^{-1} \rhoG(g_{\eta})  \left( \prod_{i=\ell}^1 \beta_i(t) \tau_{e'_i}\right)  \beta_0(t)\]
Moreover, if $\gamma_i(0)=\identityLieGroup$ for $i=0,\ldots, k$ and $\beta_i(0)=\identityLieGroup$ for $i=0,\ldots, \ell$ then $F_0=X_{\eta}$ where $X_{\eta}$ is the map induced by $x_\eta$.
\begin{proof}
 
 \begin{figure}[ht]
  \centering
    \begin{tikzpicture}[dot/.style = {circle, fill, minimum size=#1, inner sep=0pt, outer sep=0pt},dot/.default = 6pt, scale=0.5]
      \path (-3,0) node[dot](X0){};
      \path (0,0) node[dot](X1){};
      \path (2,1.5) node[dot](X2){};
      \path (5,2) node[dot](X3){};
      \path (7,4.5) node[dot](X4){};
      \path (3,-1.5) node[dot](X5){};
      \path (6,-2) node[dot](X6prime){};
      \path (8,-2) node[dot](X6){};
      \path (10,-3) node[dot](X7){};
      \path (10.5,-2.5) node[dot](gX7){};
      
      \begin{scope}[thick, decoration={ markings, mark=at position 0.5 with {\arrow{>}}},] 
      \draw (X0.center) edge[postaction={decorate}]node[midway,above ]{$e_1$} (X1.center);
      \draw (X0.center) edge[postaction={decorate}]node[midway,below left ]{$e'_1$} (X1.center);
      
      \draw (X1.center) edge[postaction={decorate}]node[midway,above left ]{$e_2$} (X2.center);
      \draw (X2.center) edge[postaction={decorate},color=gray] node[midway,above ]{$\cdots$} (X3.center);
      \draw (X3.center) edge[postaction={decorate}]node[midway,above left ]{$e_k$} (X4.center);
      \draw (X1.center) edge[postaction={decorate}]node[midway,below ]{$e'_2$} (X5.center);
      \draw (X5.center) edge[postaction={decorate},color=gray] node[midway,below ]{$\cdots $} (X6prime.center);
      \draw (X6prime.center) edge[postaction={decorate}] node[midway,below ]{$e'_{\ell-1}$} (X6.center);
      \draw (X6.center) edge[postaction={decorate}] node[midway,below ]{$e'_\ell$} (X7.center);
	\draw (X4) edge[postaction={decorate}, out=-40,in=80, looseness=1]node[midway,right]{$\eta$} (gX7);
      \end{scope}
      \node [dot] at (X0) {}; 
      \node[left] at (X0) {$v_0$};
      \node [dot] at (X1) {}; 
      \node [dot] at (X2) {}; 
      \node [dot] at (X3) {}; 
      \node [above] at (X4) {};
      \node [dot] at (X5) {}; 
      \node[dot] at (X6) {}; 
      \node[dot] at (X6prime) {}; 
      \node[right] at (gX7) {$\, t(\eta)$}; 
      \node[below right] at (X7) {$t(e'_\ell)=w(\eta)$}; 
    \end{tikzpicture}
  \caption{The paths in \cref{homotopyFromPathsInCentralizers,inducedMapIsHomotopicToMapWithFactorsConjugated}. Note that $t(e'_\ell)=w(\eta)=g_{\eta}^{-1}\cdot t(\eta)$. Also note that we may have $k=0$ or $\ell=0$.
  }\label{figureY0}
\end{figure}
  The following computation shows that $F$ is $\HH$-equivariant.
  \begin{align*}
F(\tau\cdot \alpha,t) &=  \gamma_0(t)  \left(\prod_{i=1}^k (\tau\cdot \alpha)_{e_i}^{-1} \gamma_i(t) \right)  (\tau\cdot \alpha)_{\eta}^{-1} \rhoG(g_{\eta})  \left( \prod_{i=\ell}^1 \beta_i(t) (\tau\cdot \alpha)_{e'_i}\right)  \beta_0(t)\\
   &=  \gamma_0(t)  \left(\prod_{i=1}^k \alpha_{s(e_i)}^{-1}\tau_{e_i}^{-1} \alpha_{t(e_i)} \gamma_i(t) \right)   \alpha_{s(\eta)}^{-1}\tau_{\eta}^{-1}  \rhoG(g_\eta) \alpha_{w(\eta)} \rhoG(g_\eta)^{-1} \\
   &\qquad \cdot \rhoG(g_{\eta})  \left( \prod_{i=\ell}^1 \beta_i(t) \alpha_{t(e'_i)}^{-1}\tau_{e'_i}\alpha_{s(e'_i)} \right)  \beta_0(t)\\
   &=  \gamma_0(t)  \left(\prod_{i=1}^k\tau_{e_i}^{-1} \gamma_i(t) \right) \tau_{\eta}^{-1}  \rhoG(g_\eta)  \left( \prod_{i=\ell}^1 \beta_i(t) \tau_{e'_i} \right)  \beta_0(t)\\
   &= F(\tau,t).
  \end{align*}

For the second part, note that
  \begin{align*}
  X_{\eta}(\tau)&=\rho_\tau( x_\eta ) \\ &= \tau_{s(\eta)}^{-1}\tau_{\eta}^{-1} \rhoG(g_\eta)\tau_{w(\eta)} \\
  &= \left(\prod_{i=1}^k \tau_{e_i}^{-1}\right)  \tau_{\eta}^{-1} \rhoG(g_{\eta})  \prod_{i=\ell}^1 \tau_{e'_i}.
  \end{align*}

 \end{proof}
\end{proposition}

\begin{proposition}\label{inducedMapIsHomotopicToMapWithFactorsConjugated}
Suppose that $\centralizer_{\lieGroup}(\centralizer_\lieGroup(\rhoG(G_{v})))$ is connected for each $v\in V$.
Let $\eta\in E-T$ and let $(e_1,\ldots, e_k)$ and $(e'_1,\ldots, e'_\ell)$ be the unique paths in $T$ from $v_0$ to $s(\eta)$ and $w(\eta)$ respectively.
  Let $A_e\in \lieGroup$ be elements defined for every $e\in E$.
  Suppose that $C_1,\ldots,C_k,B_1,\ldots,B_\ell \in \lieGroup$ satisfy:
\begin{itemize}
    \item $A_{e_i}^{-1} C_i A_{e_{i+1}}$ commutes with $\centralizer_\lieGroup(\rhoG(G_{t(e_i)}))$ for $i=1,\ldots,k-1$.
   \item $A_{e_k}^{-1} C_k A_{\eta}$ commutes with $\centralizer_\lieGroup(\rhoG(G_{t(e_k)}))$.
   \item $A_{e'_{i+1}}^{-1} B_{i} A_{e'_i}$ commutes with $\centralizer_\lieGroup(\rhoG(G_{t(e'_i)}))$  for $i=1,\ldots,\ell-1$.
   \item $\rhoG(g_\eta)^{-1} A_{\eta}^{-1}\rhoG(g_{\eta}) B_\ell A_{e'_\ell}$  commutes with $\centralizer_\lieGroup(\rhoG(G_{w(\eta)}))$.
\end{itemize}

Then there is an $\HH$-equivariant homotopy between the map $X_{\eta}\colon \MM\to \lieGroup$ induced by $x_\eta$ and the map $Z\colon \MM\to \lieGroup$ defined by
\[Z(\tau) = \left(\prod_{i=1}^k A_{e_i}\tau_{e_i}^{-1}A_{e_i}^{-1}C_i \right)  A_{\eta}\tau_{\eta}^{-1}A_{\eta}^{-1} \rhoG(g_{\eta})  \left( \prod_{i=\ell}^1 B_i A_{e'_i}\tau_{e'_i} A_{e'_i}^{-1}\right)_.\]

\begin{proof}
 Since the centralizers $\centralizer_{\lieGroup}(\centralizer_\lieGroup(\rhoG(G_v)))$ are connected, we can take paths:
  \begin{itemize}
   \item $\identityLieGroup \xrightarrow{\gamma_0} A_{e_1}$, in $\lieGroup$.
   \item $\identityLieGroup \xrightarrow{\gamma_i} A_{e_i}^{-1} C_i A_{e_{i+1}}$ such that $\gamma_i(I)$ commutes with $\centralizer_\lieGroup(\rhoG(G_{t(e_i)}))$ for $i=1,\ldots,k-1$.
   \item $\identityLieGroup \xrightarrow{\gamma_k} A_{e_k}^{-1} C_k A_{\eta}$ such that $\gamma_k(I)$ commutes with $\centralizer_\lieGroup(\rhoG(G_{t(e_k)}))$.
   \item $\identityLieGroup \xrightarrow{\beta_\ell} \rhoG{(g_\eta)}^{-1} A_{\eta}^{-1}\rhoG(g_{\eta}) B_\ell A_{e'_\ell}$  such that $\beta_l(I)$ commutes with $\centralizer_\lieGroup(\rhoG(G_{w(\eta)}))$.
   \item $\identityLieGroup \xrightarrow{\beta_i}  A_{e'_{i+1}}^{-1} B_{i} A_{e'_i}$ such that $\beta_i(I)$ commutes with $\centralizer_\lieGroup(\rhoG(G_{t(e'_i)}))$  for $i=1,\ldots,\ell-1$.
   \item $\identityLieGroup \xrightarrow{\beta_0} A_{e'_1}^{-1}$, in $\lieGroup$.
  \end{itemize} 
  The result now follows from \cref{homotopyFromPathsInCentralizers}.
\end{proof}
\end{proposition}


\section{Choices and notation for graph $\XOS(G)$}

Now we shall fix, for each of the groups $G$ in \cref{thmB}, our choices regarding $\XOSk(G)$ in order to apply Brown's result to it.
By \cref{HowOrbitsConnectInXOS}, we can assume that in each case the orbits are connected as in \cref{FigureXOSG}.
Our choices in each case are the following (the stabilizers are recorded in \cref{TableStabilizersVerticesOSGraph,TableStabilizersEdgesOSGraph}).
\begin{itemize}
\item For $G=\PSL_2(2^n)$
we take $V=\{v_0,v_1,v_2\}$, $E=\{\eta_0,\eta_1,\eta_2,\allowbreak\eta'_1,\ldots,\eta'_k\}$, and $T=\{\eta_0,\eta_1\}$, with
$v_0\xrightarrow{\eta_0}v_1$, $v_1\xrightarrow{\eta_1}v_2$, $v_2\xrightarrow{\eta_2} g_{\eta_2} v_0$  and $v_0\xrightarrow{\eta'_i} v_0$ for $i=1,\ldots, k$.

\item For $G=\PSL_2(q)$ with $q=3^n$ or $q\equiv 19\pmod {24}$
we take $V=\{v_0,v_1,v_2,v_3\}$, $E=\{\eta_0,\eta_1,\eta_2,\eta_3,\allowbreak\eta'_1,\ldots,\eta'_k\}$, and $T=\{\eta_0,\eta_1,\eta_2\}$, with
$v_0\xrightarrow{\eta_0}v_1$, $v_1\xrightarrow{\eta_1}v_2$, $v_2\xrightarrow{\eta_2}  v_3$, $v_3\xrightarrow{\eta_3} g_{\eta_3} v_0$  and $v_0\xrightarrow{\eta'_i} v_0$ for $i=1,\ldots, k$.

\item For $G=\PSL_2(q)$ with $q\equiv 11\pmod {24}$
we take $V=\{v_0,v_1,v_2,v_3\}$, $E=\{\eta_0,\eta_1,\eta_2,\eta_3,\allowbreak\eta'_1,\ldots,\eta'_k\}$, and $T=\{\eta_0,\eta_1,\eta_2\}$, with
$v_0\xrightarrow{\eta_0}v_1$, $v_1\xrightarrow{\eta_1}v_2$, $v_2\xrightarrow{\eta_2}  v_3$, $v_3\xrightarrow{\eta_3} g_{\eta_3} v_2$  and $v_0\xrightarrow{\eta'_i} v_0$ for $i=1,\ldots, k$.

\item For $G=\Sz(q)$
we take $V=\{v_0,v_1,v_2, v_3\}$, $E=\{\eta_0,\eta_1,\eta_2,\eta_3,\allowbreak \eta'_1,\ldots,\eta'_k\}$, and $T=\{\eta_0,\eta_1,\eta_2\}$, with
$v_0\xrightarrow{\eta_0}v_1$, $v_1\xrightarrow{\eta_1}v_2$, $v_2\xrightarrow{\eta_2} v_3$, $v_3\xrightarrow{\eta_3} g_{\eta_3}v_0$  and $v_0\xrightarrow{\eta'_i} v_0$ for $i=1,\ldots, k$.
\end{itemize}
In all cases and for $i=1,\ldots,k$ we set $g_{\eta'_i}=1$.
Note that in all cases the stabilizer of $v_0$ is a Borel subgroup of $G$.
In what follows $\Gamma_k=\pi_1(\XOSk(G),v_0)\unsplitExtension G$ is the group obtained by applying Brown's result to the action of $G$ on $\XOSk(G)$ with these choices.

\renewcommand{\arraystretch}{1.25}
\begin{table}[H]
  \centering
 \begin{tabular}{|c|c|c|c|c|c|}
   \hline
   $G$ & $q$ & $G_{v_0}$ & $G_{v_1}$ & $G_{v_2}$ & $G_{v_3}$  \\
   \hline
   $\PSL_2(q)$ & $2^n$ &$B=\F_q\rtimes C_{q-1}$ & $D_{2(q-1)}$ & $D_{2(q+1)}$ &- \\
   \hline
   $\PSL_2(q)$ & $q\equiv 3\pmod 8$ & $B=\F_q\rtimes C_{(q-1)/2}$ & $D_{q-1}$ & $D_{q+1}$ & $A_4$   \\
   \hline
   $\Sz(q)$ & $2^n$ & $B=M(q,\theta)$ & $D_{2(q-1)}$ & $C_{q+r+1}\rtimes C_4$ & $C_{q-r+1}\rtimes C_4$ \\
   \hline  
  \end{tabular}
  \caption{Stabilizers of vertices for the $G$-graph $\XOSk(G)$}
 \label{TableStabilizersVerticesOSGraph}
 \end{table}
 
 \begin{table}[H]
  \centering
 \begin{tabular}{|c|c|c|c|c|c|c|}
   \hline
   $G$ & $q$  & $G_{\eta_0}$ & $G_{\eta_1}$ & $G_{\eta_2}$ & $G_{\eta_3}$ & $G_{\eta'_i}$\\
   \hline
   $\PSL_2(q)$ & $2^n$ & $C_{q-1}$ & $C_2$& $C_2$ & - & $1$\\
   \hline
   $\PSL_2(q)$ & $q\equiv 3 \pmod 8$ & $C_{(q-1)/2}$ & $C_2$ & $C_2\times C_2$ & $C_3$ & $1$  \\
   \hline
   $\Sz(q)$ & $2^n$ &  $C_{q-1}$& $C_2$ &$C_4$ & $C_4$  & $1$\\
   \hline  
  \end{tabular}
   \caption{Stabilizers of edges for the $G$-graph $\XOSk(G)$}
 \label{TableStabilizersEdgesOSGraph}
 \end{table}
 \renewcommand{\arraystretch}{1}


\section{Representations and centralizers}
\label{sectionCentralizers}
In this section we obtain, for each of the groups $G$ in \cref{thmB}, a suitable irreducible representation $\rhoG$ of $G$ in $\lieGroup=\U(m)$.
The value of $m$ for each case is recorded in Table~\ref{tableDimensionRho}.
We start by recalling the following classical results

\renewcommand{\arraystretch}{1.25}
\begin{table}[ht]
 \centering
\begin{tabular}{|c|c|c|}
\hline
 $G$ &   $q$ & $m$ \\
 \hline
 $\PSL_2(q)$  &$2^n$ &$q-1$ \\
 \hline
 $\PSL_2(q)$  &$3^n$ with $n$ odd & $(q-1)/2$ \\
 \hline
 $\PSL_2(q)$  & $q\equiv 11\text{ or }19\pmod{24}$ & $(q-1)/2$ \\
 \hline
 $\Sz(q)$ & $2^n$ with $n$ odd & $(q-1){\sqrt{2q}}/2$ \\
 \hline
\end{tabular}
\caption{The degree $m$ of $\rhoG$ in each case.}
\label{tableDimensionRho}
\end{table}
\renewcommand{\arraystretch}{1}

\begin{theorem}[{\cite[Theorem 4.6.2]{IntroRepTheory}}]\label{RepresentationIsEquivalentToUnitaryRepresentation}
 Every representation $\rho\colon G\to \GL_n(\C)$ of a finite group $G$ is isomorphic to a unitary representation $\wt{\rho}\colon G\to \U(n)$.
\end{theorem}

\begin{theorem}
\label{theoremMalcev}
 Let $G$ be a finite group.
 If two unitary representations of $G$ are isomorphic then there is a unitary isomorphism between them.
  \begin{proof}
   When the representations are irreducible this is~\cite[Lemma 33.1]{DornhoffPartA}.
   For a proof in the general case see~\cite{UnitaryEquivalenceMO}.
  \end{proof}
\end{theorem}

If $A,A'$ are matrices then $A\oplus A'$ denotes the block diagonal matrix $\begin{pmatrix}A & 0 \\ 0 & A'\end{pmatrix}$.
If $\rho,\rho'$ are representations of a group $G$ then $\rho\oplus \rho'$ denotes the representation such that $(\rho\oplus\rho')(g)= \rho(g)\oplus\rho'(g)$ for all $g\in G$.
We denote the $n\times n$ identity matrix by $I_n$.
It is easy to verify that block scalar matrices commute with scalar block matrices:
\begin{proposition}\label{ScalarBlockMatricesCommuteWithBlockScalarMatrices}
Let $X\in M_n(\C)$ and $\lambda \in M_k(\C)$ be two matrices.
Let $A=X\oplus\cdots \oplus X\in M_{kn}(\C)$ and let $B\in M_{kn}(\C)=M_k(M_n(\C))$ be the matrix defined by $B_{i,j}=\lambda_{i,j}I_n$.
Then $A$ and $B$ commute.
\end{proposition}

\begin{remark}\label{remarkSchur}
Let $\rho_1,\ldots,\rho_k$ be pairwise non-isomorphic irreducible representations of a finite group $G$ and let $n_1,\ldots,n_k$ be natural numbers. 
Consider the representation $\rho = \bigoplus_{i=1}^k \rho_i^{n_i}$,
where $\rho_i^{n_i}$ denotes the sum $\rho_i \oplus \cdots \oplus \rho_i$ of $n_i$ copies of $\rho_i$.
Let $d_i$ be the degree of $\rho_i$ and let $n=\sum_{i=1}^k d_i n_i$ be the degree of $\rho$.
Then, by Schur's lemma, we have
\[\centralizer_{\U(n)}( \rho(G)) = \prod_{i=1}^k \centralizer_{ \U(d_in_i) }\left( \rho_i^{n_i} \right),\]
where the product on the right is included in $\U(n)$ as block diagonal matrices.
Again by Schur's lemma, we have an isomorphism
$\U(n_i)\xrightarrow{\,\, \groupIso\,\, } \centralizer_{ \U(d_in_i) }\left( \rho_i^{n_i} \right)$
which is given by $A\mapsto \wt{A}$, where $\wt{A}\in M_{n_id_i}(\C)=M_{n_i}(M_{d_i}(\C))$ is the scalar block matrix defined by $\wt{A}_{s,t}= A_{s,t}I_{d_i}$, which is in fact unitary.
Then $\centralizer_{\U(n)}( \rho(G))\groupIso \prod_{i=1}^k \U(n_i)$ and in particular we have $\dim \centralizer_{\U(n)}(\rho(G))=\sum_{i=1}^k n_i^2$.
\end{remark}

\begin{lemma}\label{lemmaDimensionCentralizerU}
Let $G$ be a finite group and let $\rho\colon G\to \U(n)$ be a unitary representation with character $\chi$.
Then $\dim\ \centralizer_{\U(n)}(\rho(G))  ={\left\langle \chi, \chi \right\rangle}_G$.
\begin{proof}
 If $\rho$ is isomorphic to $\bigoplus_{i=1}^k \rho_i^{n_i}$, where $\rho_1,\ldots, \rho_k$ are pairwise non-isomorphic irreducible representations of $G$, 
 from the orthogonality relations and \cref{remarkSchur} we obtain $\langle \chi,\chi\rangle_G = \sum_{i=1}^k n_i^2 = \dim \centralizer_{\U(n)}(\rho(G))$.
\end{proof}
\end{lemma}

In what follows, $(x)$ denotes the conjugacy class of $x\in G$.

\begin{proposition}\label{repPow2}
  Let $G=\PSL_2(q)$ with $q=2^n$.
  Then there are elements $a,b,c\in G$ with orders $|a|=q-1$, $|b|=q+1$ and $|c|=2$, such that the following hold:
  \begin{enumerate}[label=(\roman*)]
    \item There are exactly $q+1$ conjugacy classes in $G$:
    $(1)$, $(a^i)$ for $1\leq i \leq q/2-1$, $(b^j)$ for $1\leq j\leq q/2$, and $(c)$.
    \item The elements in a Borel subgroup $B=\F_q\rtimes C_{q-1}$ are the following: $1$; $2q$ elements in $(a^i)$, for each $i$; and $q-1$ elements in $(c)$.
    \item  There is an irreducible character $\chi$ given by
    \begin{center}
      \begin{tabular}{c|ccccc}
               & $1$   & $(a^i)$ & $(b^j)$                   & $(c)$ \\ \hline 
        $\chi$ & $q-1$ & $0$     & $-(\omega^j+\omega^{-j})$ & $-1$
    \end{tabular}
  \end{center}
  where $\omega=e^{\frac{2\pi \ii}{q+1}}$.
  \end{enumerate}

  \begin{proof}
  The description of the conjugacy classes and the character table for $\PSL_2(2^n)=\SL_2(2^n)$ can be found in \cite[Theorem 38.2]{DornhoffPartA}.
  Note that $(c)$ is the unique conjugacy class of involutions
  and any cyclic group of order $q-1$ must contain $1$ and $2$ elements from $(a^i)$ for each $i$.
  Let $A(x,a)=\begin{pmatrix} x & a \\ 0 & x^{-1} \end{pmatrix}$.
  Then the subgroup $B=\left\{A(x,a) : a\in \F_q, x\in \F_q^*\right\}$ of upper triangular matrices is a Borel subgroup of $G$
  and $B$ is the semidirect product of $N=\{A(1,a): a\in \F_q \}\groupIso \F_q$ and $K=\{A(x,0): x\in \F_q^*\}\groupIso C_{q-1}$.
  The subgroup $N$ contains $1$ and $q-1$ involutions which must lie in $(c)$.
  A straightforward computation proves that nonidentity elements of $N$ do not commute with nonidentity elements of $K$.
  Then considering the conjugation action of $N$ on $B$ we prove part (ii).
  \end{proof}
\end{proposition}

\begin{proposition}\label{rep3Mod4}
  Let $G=\PSL_2(q)$ with $q=p^n \equiv 3 \pmod 4$ and $p$ prime.
  Then there are elements $a,b,c \in G$ with orders $|a|=(q-1)/2$, $|b|=(q+1)/2$ and $|c|=p$, such that the following hold:
  \begin{enumerate}[label=(\roman*)]
    \item There are exactly $(q-1)/2+3$ conjugacy classes in $G$:
    $(1)$, $(a^i)$ for $1\leq i \leq (q-3)/4$, $(b^j)$ for $1\leq j \leq (q+1)/4$, $(c)$, and $(c^{-1})$.
    \item The elements in a Borel subgroup $B=\F_q\rtimes C_{(q-1)/2}$ are the following: $1$; $2q$ elements in $(a^l)$, for $1\leq l \leq (q-3)/4$;
    $(q-1)/2$ elements in $(c)$; and $(q-1)/2$ elements in $(c^{-1})$.
    \item There is an irreducible character $\chi$ given by
    \begin{center}
      \begin{tabular}{c|ccccc}
               & $1$       & $(a^i)$ & $(b^j)$      & $(c)$                & $(c^{-1})$                \\ \hline 
        $\chi$ & $(q-1)/2$ & $0$     & $(-1)^{j+1}$ & $(-1+\sqrt{q}\ii)/2$ & $(-1-\sqrt{q}\ii)/2$ \\
    \end{tabular}
  \end{center}
  \end{enumerate}
  \begin{proof}
    In $\SL_2(q)$ there are elements $\wt{a},\wt{b},\wt{c},\wt{d}$ with $|\wt{a}|=q-1$, $|\wt{b}|=q+1$, $|\wt{c}|=|\wt{d}|=p$
    and the conjugacy classes in $\SL_2(q)$ are $1$, $-1$, $(\wt{a}^i)$ for $1\leq i \leq (q-3)/2$,
    $(\wt{b}^j)$ for $1\leq j \leq (q-1)/2$, $(\wt{c})$, $(\wt{d})$, $(-\wt{c})$, and $(-\wt{d})$ \cite[Theorem 38.1]{DornhoffPartA}
    (this holds whenever $q$ is a power of an odd prime $p$).
    The center of $\SL_2(q)$ is $\{1,-1\}$.
    Note that $-1$ is the only involution in $\SL_2(q)$.
    Since $q\equiv 3\pmod 4$, by \cite[p. 234]{DornhoffPartA} we have $\wt{c}^{-1}\in (\wt{d})$.
    Moreover $\wt{a}^{(q-1)/2}=\wt{b}^{(q+1)/2}=-1$.
    Therefore part (i) follows by considering the classes $a,b,c\in \PSL_2(q)$ of $\wt{a}$, $\wt{b}$, and $\wt{c}$.
    Note that $(b^{(q+1)/4})$ is the only class of involutions in $G$.
    Any cyclic group of order $(q-1)/2$ contains $1$ and two elements in $(a^i)$ for each $i$.
    Any subgroup isomorphic to $\F_q$ contains $1$ and half of the remaining $q-1$ elements must belong to each of the classes $(c)$ and $(c^{-1})$.
    To complete the proof of part (ii) we use the same argument we used to prove part (ii) of \cref{repPow2}.
    Finally, the character $\chi$ in part (iii) is obtained by passing to the quotient the character $\eta_1$ of \cite[Theorem 38.1]{DornhoffPartA}.
  \end{proof}
\end{proposition}

\begin{proposition}\label{repSuzuki}
  Let $G=\Sz(q)$ with $q=2^n$ and $n\geq 3$ odd. Let $r=\sqrt{2q}$.
  Then there are elements $\sigma,\rho,x,y,z \in G$ with orders $|\sigma|=2$, $|\rho|=4$, $|x|=q-1$, $|y|=q+r+1$ and $|z|=q-r+1$, such that the following hold:
  \begin{enumerate}[label=(\roman*)]
    \item There are exactly $q+3$ conjugacy classes in $G$:
    $(1)$, $(\sigma)$, $(\rho)$, $(\rho^{-1})$, $(x^i)$ for $1\leq i \leq q/2-1$,
    $(y^j)$ for $1\leq j \leq (q+r)/4$, and $(z^k)$ for $1\leq k\leq (q-r)/4$.
    \item The elements in a Borel subgroup $B$ are the following: $1$, $q-1$ elements in $(\sigma)$, $q(q-1)/2$ elements in $(\rho)$, $q(q-1)/2$ elements in $(\rho^{-1})$
    and $2q^2$ elements in $(x^i)$ for $1\leq i \leq q/2-1$.
    \item There is an irreducible character $\chi$ given by
    \begin{center}
      \begin{tabular}{c|ccccccc}
               & $1$        & $(\sigma)$ & $(\rho)$  & $(\rho^{-1})$ & $(x^i)$ & $(y^j)$ & $(z^k)$ \\ \hline 
        $\chi$ & $(q-1)r/2$ & $-r/2$     & $\ii r/2$ & $-\ii r/2$    & $0$     & $1$     & $-1$
    \end{tabular}
  \end{center}
  \end{enumerate}
  \begin{proof}
    The description of the conjugacy classes is given in \cite[\S 17]{Suzuki}.
    Any cyclic group of order $q-1$ contains $1$ and two elements in the class $(x^i)$ for each $i$.
    Recall that $B=M(q,\theta)=S(q,\theta)\rtimes T$.
    By \cite[Lemma 1]{Suzuki} the group $S(q,\theta)$ consists of: $1$,
    $q-1$ elements of order $2$ (which must be in $(\sigma)$),
    and the remaining $q^2-q$ elements have order $4$.
    Then there must be $(q^2-q)/2$ elements in each of the classes $(\rho)$ and $(\rho^{-1})$.
    Now $T$ is a cyclic group of order $q-1$ and
    since nonidentity elements of $S(q,\theta)$ do not commute with nonidentity elements of $T$ \cite[Lemma 5]{Suzuki},
    part (ii) can be obtained by considering the conjugation action of $S(q,\theta)$ on $B$.
    The character table of $\Sz(q)$ can be found in \cite[Theorem 13]{Suzuki}.
  \end{proof}
\end{proposition}

\begin{remark}\label{remarkOneConjugacyClassSmallCyclicSubgroups}
  Each of the groups $G$ in \cref{OSThmA} has at most one conjugacy class of cyclic subgroups of order $k$ for $k=2,3,4$.
  This follows from part (i) of \cref{repPow2,rep3Mod4,repSuzuki}
  (and when $q\equiv 1 \pmod 4$ from the first sentence in the proof of \cref{rep3Mod4}).
\end{remark}

\begin{proposition}\label{existenceSuitableRep}
  Let $G$ be one of the groups in \cref{thmB} and let $\lieGroup=\U(m)$ as in Table~\ref{tableDimensionRho}.
  There is an irreducible representation $\rhoG\colon G \to \lieGroup$ together with elements $A, C\in \lieGroup$ which satisfy the following properties:
  \begin{enumerate}[label=(\roman*)]
    \item The restriction of $\rhoG$ to the Borel subgroup $G_{v_0}$ is irreducible.
    \item $\dim \centralizer_{\lieGroup}(\rhoG(G_{v_1})) < \dim \centralizer_{\lieGroup}(\rhoG(G_{\eta_0}))$. 
    \item $A^{-1}\centralizer_{\lieGroup}(\rhoG(G_{\eta_0}))A \subseteq \centralizer_{\lieGroup}(\rhoG(G_{\eta_1}))$.
    \item $C\in \centralizer_{\lieGroup}(\rhoG(G_{\eta_1}))$.
    \item $A C$ commutes with $\centralizer_{\lieGroup}(\rhoG(G_{v_1}))$.
  \end{enumerate}
\begin{proof}
  Let $x$ be a generator of $G_{\eta_0}$ and $y$ be a generator of $G_{\eta_1}$.
  Note that in all cases $\langle x,y\rangle = G_{v_1}$ is a dihedral group of order $2|G_{\eta_0}|$ and $|G_{\eta_0}|$ is odd.
  In what follows $J_n$ denotes the $n\times n$ matrix with $1$ in the antidiagonal and $0$ elsewhere.  
  We also consider the matrix
  \[M_{2n}=\frac{\sqrt{2}}{2}\begin{pmatrix} I_n & J_n \\ I_n & -J_n\end{pmatrix}\in \U(2n)\]
  which satisfies
  \[ J_{2n}  = M_{2n}^{-1} (I_n \oplus -I_n) M_{2n}.\]

  For $G=\PSL_2(q)$ and $q=2^n$ we take $\rhoG$ realizing the degree $q-1$ irreducible character $\chi$ in part (iii) of \cref{repPow2}.
  By \cref{RepresentationIsEquivalentToUnitaryRepresentation}, we can take $\rhoG$ to be unitary.
  To prove (i) we compute the norm of the restriction of $\chi$ to the Borel subgroup using part (ii) of \cref{repPow2}.
  Now by part (i) of \cref{repPow2} the restriction of $\chi$ to $G_{v_1}$ is given by
  \begin{center}
    \begin{tabular}{c|ccc}
             & $1$   & $x^i$ & $x^iy$ \\ \hline 
      $\chi$ & $q-1$ & $0$   & $-1$ 
    \end{tabular}
  \end{center}
  and therefore by \cref{theoremMalcev},
  letting $\xi = e^{\frac{2\pi \ii}{q-1}}$ we can assume that
  \begin{align*}
    \rhoG(x) &= \diag(\xi, \xi^2, \ldots, \xi^{q-2}, 1) \\
    \rhoG(y) &= J_{q-2}\oplus -I_1.
  \end{align*}
  Note that $\centralizer_{\lieGroup}(\rhoG(G_{\eta_0}))= \U(1)^{q-1}$ has dimension $q-1$.
  Now let $A = M_{q-2}\oplus I_1 \in \lieGroup$. Then
  $ \rhoG(y)  = A^{-1} (I_{q/2-1}\oplus -I_{q/2})A$.
  Therefore
  \begin{align*}
    \centralizer_{\lieGroup}(\rhoG(G_{\eta_1})) &= A^{-1} \left(\U\left( q/2-1 \right)\times \U\left( q/2 \right)\right)A \\
    &\supseteq  A^{-1}\U(1)^{q-1}A  \\
    &= A^{-1}\centralizer_{\lieGroup}(\rhoG(G_{\eta_0})) A
  \end{align*}
  and (iii) follows.
  Let $C=A^{-1}(I_{q/2-1}\oplus J_{q/2-1}\oplus I_1)A$.
  Clearly (iv) holds. An easy computation shows that
  \[
  \centralizer_{\lieGroup}(\rhoG(G_{v_1})) =\{
    a_1\oplus a_2 \oplus \cdots \oplus a_{q/2-1} \oplus a_{q/2-1} \oplus \cdots \oplus a_2\oplus a_1 \oplus b
    \tq a_1,a_2,\ldots, a_{q/2-1}, b \in \U(1)
  \}.\]
  Therefore $\dim \centralizer_{\lieGroup}(\rhoG(G_{v_1})) = q/2$ and (ii) is verified.
  Finally (v) follows from a straightforward computation.

  For $G=\PSL_2(q)$ where $q\equiv 3\pmod 8$ and $q>3$ we obtain $\rhoG$ realizing the degree $\frac{q-1}{2}$ character $\chi$ in part (iii) of \cref{rep3Mod4}.
  By \cref{RepresentationIsEquivalentToUnitaryRepresentation}, we can take $\rhoG$ to be unitary.
  To prove (i) we compute the norm of the restriction of $\chi$ to the Borel subgroup using part (ii) of \cref{rep3Mod4}.
  Now by part (i) of \cref{rep3Mod4} the restriction of $\chi$ to $G_{v_1}$ is given by
  \begin{center}
    \begin{tabular}{c|ccc}
             & $1$   & $x^i$ & $x^iy$ \\ \hline 
      $\chi$ & $(q-1)/2$ & $0$   & $1$ 
    \end{tabular}
  \end{center}
  and therefore by \cref{theoremMalcev},
  letting $\xi=e^{\frac{2\pi \ii}{(q-1)/{2}}}$ we can assume that
  \begin{align*}
    \rhoG(x) &= \diag(1, \xi, \xi^2, \ldots, \xi^{(q-3)/2}) \\
    \rhoG(y) &= I_1\oplus J_{(q-3)/2}.
  \end{align*}
  Note that $\centralizer_{\lieGroup}(\rhoG(G_{\eta_0}))= \U(1)^{(q-1)/2}$ has dimension $(q-1)/2$.
  Now let $A = I_1 \oplus M_{(q-3)/2} \in \lieGroup$. Then
  $ \rhoG(y)  = A^{-1} (I_{(q+1)/4}\oplus -I_{(q-3)/4})A$.
  Therefore
  \begin{align*}
    \centralizer_{\lieGroup}(\rhoG(G_{\eta_1})) &= A^{-1} (\U\left((q+1)/4\right)\times \U\left((q-3)/4\right))A \\
    &\supseteq  A^{-1}\U(1)^{(q-1)/2}A  \\
    &= A^{-1}\centralizer_{\lieGroup}(\rhoG(G_{\eta_0})) A
  \end{align*}
  and (iii) follows.
  Let $C=A^{-1}(I_{1}\oplus I_{(q-3)/4}\oplus J_{(q-3)/4})A$.
  Clearly (iv) holds. An easy computation shows that
  \[
  \centralizer_{\lieGroup}(\rhoG(G_{v_1})) =\{
    b\oplus a_1\oplus a_2 \oplus \cdots \oplus a_{(q-3)/4} \oplus a_{(q-3)/4} \oplus \cdots \oplus a_2\oplus a_1
    \tq b, a_1,a_2,\ldots, a_{(q-3)/4} \in \U(1)
  \}.\]
  Therefore  $\dim \centralizer_{\lieGroup}(\rhoG(G_{v_1})) = (q+1)/4$ and (ii) is verified.
  Finally (v) follows from a straightforward computation.

  For $G=\Sz(q)$ with $q=2^n$ and $n\geq 3$ odd, let $r=\sqrt{2q}$.
  We take $\rhoG$ realizing the degree $(q-1)r/2$ character $\chi$ in part (iii) of \cref{repSuzuki}.
  By \cref{RepresentationIsEquivalentToUnitaryRepresentation}, we can take $\rhoG$ to be unitary.
  To prove (i) we compute the norm of the restriction of $\chi$ to the Borel subgroup using part (ii) of \cref{repSuzuki}.
  Now by part (i) of \cref{repSuzuki} the restriction of $\chi$ to $G_{v_1}$ is given by
  \begin{center}
    \begin{tabular}{c|ccc}
             & $1$   & $x^i$ & $x^iy$ \\ \hline 
      $\chi$ & $(q-1)r/2$ & $0$   & $-r/2$ 
    \end{tabular}
  \end{center}
  and therefore by \cref{theoremMalcev},
  letting $\xi = e^{\frac{2\pi \ii}{q-1}}$ we can assume that
  \begin{align*}
    \rhoG(x) &= \xi I_{r/2} \oplus \xi^2 I_{r/2} \oplus\cdots \oplus \xi^{q-2}I_{r/2} \oplus I_{r/2} \\
    \rhoG(y) &= J_{(q-2) r/2}\oplus -I_{r/2}.
  \end{align*}
  Note that $\centralizer_{\lieGroup}(\rhoG(G_{\eta_0}))= \U(r/2)^{q-1}$ has dimension $(q-1)q/2$.
  Now let $A = M_{(q-2) r/2}\oplus I_{r/2} \in \lieGroup$. Then
  $ \rhoG(y)  = A^{-1} (I_{(q/2-1) r/2}\oplus -I_{q/2 \cdot r/2})A$.
  Therefore
  \begin{align*}
    \centralizer_{\lieGroup}(\rhoG(G_{\eta_1})) &= A^{-1} (\U\left((q/2-1) r/2\right)\times \U\left(q/2\cdot r/2 \right))A \\
    &\supseteq  A^{-1}\U(r/2)^{q-1}A  \\
    &= A^{-1}\centralizer_{\lieGroup}(\rhoG(G_{\eta_0})) A
  \end{align*}
  and (iii) follows.
  Let $C=A^{-1}(I_{(q/2-1) r/2}\oplus J_{(q/2-1)r/2}\oplus I_1)A$.
  Clearly (iv) holds. An easy computation shows that
  \[
  \centralizer_{\lieGroup}(\rhoG(G_{v_1})) =\{
    a_1\oplus a_2 \oplus \cdots \oplus a_{q/2-1} \oplus a_{q/2-1} \oplus \cdots \oplus a_2\oplus a_1 \oplus b
    \tq a_1,a_2,\ldots, a_{q/2-1}, b \in \U(r/2)
  \}.\]
  Therefore  $\dim \centralizer_{\lieGroup}(\rhoG(G_{v_1})) = q^2/4$ and (ii) is verified.
  Finally (v) follows from a straightforward computation.
\end{proof}
\end{proposition}


\section{The dimension of $\MMbar_k$}

From now on, let $\MM_k$ be the moduli of representations of $\Gamma_k$ obtained by
applying the construction of \cref{constructionModuli} to a representation $\rhoG$ obtained using \cref{existenceSuitableRep}. 
Let $\MMbar_k$ be the corresponding quotient obtained using \cref{quotientModuli}.
Note that $\MM_k = \MM_0 \times \lieGroup^k$ and that $\MMbar_k=\MMbar_0\times \lieGroup^k$.
From \cref{connectednessOfMAndMbar} we know that $\MM_k$ and $\MMbar_k$ are connected and orientable.

This section is devoted to proving that $\dim \MMbar_k = \dim \lieGroup^{k+1} $.
It is straightforward but tedious to prove this by expressing in terms of powers of $q$ the dimension of each centralizer involved in the definition of $\MMbar_k$
for this can be done by restricting the character of $\rhoG$ to each stabilizer subgroup and computing the norm of the restriction.
We present an alternative, more elegant proof which sheds light on why this equality holds in all cases.

\begin{lemma}[Piterman]\label{AnotherProofDimMM}
  Let $X$ be an acyclic $2$-dimensional $G$-complex and let $\varphi,\psi$ be two characters of $G$.
  Let $V$, $E$, $F$ be representatives of the orbits of vertices, edges and $2$-cells in $X$.
  Then  
  \[\langle \varphi,\psi\rangle_G + \sum_{e\in E} \langle \Res^G_{G_e}\varphi , \Res^G_{G_e}\psi \rangle_{G_e} 
      = \sum_{v\in V} \langle \Res^G_{G_v}\varphi , \Res^G_{G_v}\psi \rangle_{G_v}  + \sum_{f\in F} \langle \Res^G_{G_f}\varphi , \Res^G_{G_f}\psi \rangle_{G_f}.\]
  
  \begin{proof}
    Since $X$ is acyclic, $\wt{C}_{-1}(X;\C)\oplus\wt{C}_1(X;\C) \simeq \wt{C}_0(X;\C)\oplus \wt{C}_2(X;\C)$ as $G$-modules.
    Then, letting $\alpha_H$ be the character of the $G$-module $\C[G/H]$ we have
    \[\alpha_G + \sum_{e\in E} \alpha_{G_e}  = \sum_{v\in V} \alpha_{G_v} + \sum_{f\in F} \alpha_{G_f}\]
    and now
    the result follows from Frobenius reciprocity:
    \[\langle \Res^G_H\varphi, \Res^G_H\psi\rangle_H= \langle \varphi, \Ind^G_H \Res^G_H\psi \rangle_G = \langle \varphi, \alpha_H\psi\rangle_{G.}\]
  \end{proof}
\end{lemma}

Combining this with \cref{lemmaDimensionCentralizerU} we obtain:

\begin{corollary}\label{coroDimensionCentralizers}
  Let $X$ be an acyclic $2$-dimensional $G$-complex and let $\rho\colon G\to \U(n)$ be a unitary representation.
  Let $V$, $E$, $F$ be representatives of the orbits of vertices, edges and $2$-cells in $X$.
  Then
  \[
    \dim \centralizer_{\U(n)}(\rho(G)) 
    - \sum_{v \in V} \dim \centralizer_{\U(n)}(\rho(G_v)) 
    + \sum_{e \in E} \dim  \centralizer_{\U(n)}(\rho(G_e))
    - \sum_{f \in F}\dim \centralizer_{\U(n)}(\rho(G_f))
    = 0.
  \]
\end{corollary}

\begin{proposition}\label{propositionDimensionMM0}
  For each of the groups $G$ in \cref{thmB}, the dimension of $\MMbar_k$ is equal to the dimension of $\lieGroup^{k+1}$.
  \begin{proof}
    We consider an acyclic $2$-complex $X$ obtained from $\XOSk(G)$ by attaching $k+1$ free orbits of $2$-cells.
    We apply \cref{coroDimensionCentralizers}.
    By \cref{existenceSuitableRep}, $\rhoG$ and its restriction to the Borel subgroup $G_{v_0}$ are irreducible and we thus have
    $\dim \centralizer_{\lieGroup}(\rho(G)) = \dim \centralizer_{\lieGroup}(\rho(G_{v_0})) = 1$ so these terms cancel.
    Moreover for $f\in F$ we have $\centralizer_{\lieGroup}(\rho(G_f))=\lieGroup$ and so by part (iii) of \cref{quotientModuli} we are done.
  \end{proof}
\end{proposition}


\section{The differential of $\WWbar$ at $\identityMMBar$} \label{sectionDifferentialX}

For each of the groups $G$ in \cref{thmB},
we consider a closed edge path $\xi$ in $\XOS(G)$ such that attaching a free orbit of $2$-cells to it along this path gives an acyclic $2$-complex.
We define $x_0=\inclusionBrown(\xi)$, where $\inclusionBrown\colon \pi_1(\XOS(G),v_0)\to \Gamma_0$ is the inclusion given by Brown's theorem.
We set $x_i=x_{\eta'_i}$ for $i=1,\ldots,k$.
For $i=0,\ldots,k$ we consider the map $X_i\colon\MM\to \lieGroup$ induced by $x_i$.

We explain here some notation which is only needed in this proof of the following lemma.
If $x=\sum_{g\in G}x_g g\in \Z[G]$ then we define $\overline{x}=\sum_{g\in G} x_g g^{-1}$.
We have $\overline{x+y}=\overline{x}+\overline{y}$ and $\overline{x\cdot y}=\overline{y}\cdot \overline{x}$.
If $H\leq G$ is a subgroup we define $N(H)=\sum_{h\in H} h$.

\begin{lemma}\label{linearCombinationFixedPoints}
 Let $G$ be one of the groups in \cref{OSThmA}.
 Let $E$ be a set of representatives of the orbits of edges in $\XOS(G)$.
 Let $X$ be an acyclic $2$-complex obtained from $\XOS(G)$ by attaching a free orbit of $2$-cells along (the orbit of) a closed edge path
 $\xi = (a_1 e_1^{\varepsilon_1}, \ldots, a_n e_n^{\varepsilon_n})$
 with $e_i\in E$, $a_i\in G$ and $\varepsilon_i\in\{-1,1\}$.
 Let $G_e$ be the stabilizer of $e$.
 Then it is possible to choose, for each $e\in E$ an element $x_e\in \Z[G]$ such that
 \[1 = \sum_{i=1}^n \varepsilon_i a_i N(G_{e_i})x_{e_i}.\]
 Therefore for any representation $V$ of $G$ we have $V= \displaystyle\sum_{e\in E} s_e V^{G_e}$, 
 where $\displaystyle s_e = \sum_{i\in I_e} \varepsilon_i a_i$ and $I_e=\{i\tq e_i = e\}$.
  \begin{proof}
  We consider the cellular chain complex of $X$ (which is a complex of left $\Z[G]$-modules).
  Let $\alpha$ be the $2$-cell attached along $\xi$.
  We have isomorphisms $C_2(X)\simeq \Z[G]$ and $C_1(X)\simeq \displaystyle \bigoplus_{e\in E} \Z[G/G_e]$ given by $\alpha\mapsto 1$ and $e\mapsto 1\cdot G_e$ respectively.
  With these identifications, the differential $d_2\colon C_2(X)\to C_1(X)$ is given by $d_2(1) = \displaystyle\sum_{i=1}^n \varepsilon_i a_i G_{e_i} = \sum_{e\in E} s_e G_e$.
  Now the differential $d^2\colon C^1(X;\Z)\to C^2(X;\Z)$ identifies with the map 
  \begin{align*}
  d^2\colon \bigoplus_{e\in E} \Z[G/G_e]&\to \Z[G]\\
   1\cdot G_e &\mapsto N(G_e) \overline{s_e}.
   \end{align*}
  Since $X$ is acyclic, the differential $d^2$ is surjective and there are elements $y_e\in \Z[G]$ such that
  $1=\sum_{e\in E} y_e N(G_e) \overline{s_e}$.
  Finally, since $\overline{N(H)}=N(H)$ and letting $x_e=\overline{y_e}$ we have $1=\sum_{e\in E} s_e N(G_e) x_e$.
  \end{proof}  
\end{lemma}

\begin{proposition}\label{lemmaRegularPointXX}
 Let $\XX = (X_0,\ldots,X_k)\colon \MM_k\to \lieGroup^{k+1}$.
 Then $\identityMM$ is a regular point of $\XX$.
\begin{proof}
 The proof reduces to the case of $k=0$.
 Consider the representation
 \[\Ad\circ \rhoG \colon G\to \GL( T_\identityLieGroup \lieGroup )\]
 which (by \cref{DescriptionAd}) is given by $g\cdot v = d_{\rhoG{(g)}^{-1}} L_{\rhoG(g)}\circ d_\identityLieGroup R_{\rhoG(g)^{-1}} ( v )$.
 By \cref{LieAlgebraOfCentralizer} we have $T_\identityLieGroup \centralizer_\lieGroup( \rhoG( H ) ) = {(T_\identityLieGroup \lieGroup)}^H$.
 Then by \cref{linearCombinationFixedPoints} we have $T_\identityLieGroup \lieGroup = \sum_{e\in E} s_e \cdot T_\identityLieGroup \centralizer_\lieGroup(\rhoG(G_e ))$.
 Now the result follows from \cref{differentialWordInducedFromClosedEdgePath}.
\end{proof}
\end{proposition}

\begin{proposition}
If $w_0,\ldots, w_k\in N$ satisfy $N=\llangle w_0,\ldots, w_k\rrangle^{\Gamma_k} [N,N]$, then $\identityLieGroup$ is a regular point of $\WW=(W_0,\ldots,W_k)\colon \MM_k \to \lieGroup^{k+1}$.
\begin{proof}
This follows from \cref{theoremDegreeW} and \cref{lemmaRegularPointXX}.
\end{proof}
\end{proposition}

Now since $\WWbar\circ p = \WW$ we have:

\begin{corollary}\label{universalRepIsRegularPoint}
 If $w_0,\ldots, w_k\in N$ satisfy $N=\llangle w_0,\ldots, w_k\rrangle^{\Gamma_k} [N,N]$, then $\identityMMBar$ is a regular point of $\WWbar=(\Wbar_0,\ldots,\Wbar_k)\colon \MMbar_k \to \lieGroup^{k+1}$.
\end{corollary}


\section{The degree of $\WWbar$}
\label{sectionDegree}
In this section we prove the degree of $\WWbar$ is $0$.
We start by recalling the definition and some properties of the degree (see e.g.~\cite[Chapter 17]{Lee} for a detailed exposition).
Let $M$ and $M'$ be oriented $m$-manifolds.
The \textit{degree} $\deg(f)$ of a smooth map $f\colon M\to M'$ is the unique integer $k$ such that 
\[\int_M f^*(\omega) = k\int_{M'} \omega\]
for every smooth $m$-form $\omega$ on $M'$.
If $x\in M$ is a regular point of $f$, then $d_xf\colon T_x M\to T_{f(x)} M'$ is an isomorphism between oriented vector spaces and we can consider its sign $\sg\, d_x f$.
If $y\in M'$ is a regular value of $f$ we have \[\deg(f)=\sum_{x\in f^{-1}(y)}\sg\, d_x f.\]
In particular if $f$ is not surjective then $\deg(f)=0$.
Homotopic maps have the same degree.
If $N$ and $N'$ are oriented $n$-manifolds and $g\colon N\to N'$ is a smooth map then $\deg( f\times g ) = \deg(f) \deg(g)$.
If $M''$ is an oriented $m$-manifold and $h\colon M'\to M''$ is smooth then $\deg(h\circ f)=\deg(h) \deg(f)$.

Note that in all cases there is a unique edge of $\XOS(G)$ which lies in $E-T$, which we denote by $\onlyEdgeInEMinusT$.
We define $y_0=x_{\onlyEdgeInEMinusT}$ and $y_i=x_{\eta'_i}$ for $i=1,\ldots,k$.
For $i=0,\ldots,k$ we consider the map $Y_i\colon \MM \to \lieGroup$ induced by $y_i$.
The proof that $\deg(\WWbar)=0$ reduces to proving $\deg(\YYbar)=0$ which in turn reduces to proving that $\overline{Y}_0\colon \MMbar_0\to \lieGroup$ has degree $0$.
\cref{TableY0} gives the value of $Y_0$ in the different cases that we consider.
We will also consider two auxiliary maps $T$ and $Z$.

\renewcommand{\arraystretch}{1.25}
\begin{table}[H]
 \centering
  \begin{tabular}{|c|c|c|}
  \hline
  $G$ & $q$ & $Y_0(\tau)$ \\
  \hline
  $\PSL_2(q)$ & $2^n$ & $\tau_{\eta_0}^{-1}\tau_{\eta_1}^{-1}\tau_{\eta_2}^{-1}\rhoG(g_{\eta_2})$\\
  \hline
  $\PSL_2(q)$ & $3^n$ & $\tau_{\eta_0}^{-1}\tau_{\eta_{1}}^{-1}\tau_{\eta_{2}}^{-1}\tau_{\eta_3}^{-1}\rhoG(g_{\eta_3})$ \\
  \hline
  $\PSL_2(q)$ & $q\equiv 19 \pmod {24}$  & $\tau_{\eta_0}^{-1}\tau_{\eta_{1}}^{-1}\tau_{\eta_{2}}^{-1}\tau_{\eta_3}^{-1}\rhoG(g_{\eta_3})$\\
  \hline
  $\PSL_2(q)$ & $q\equiv 11\pmod {24}$  & $\tau_{\eta_0}^{-1}\tau_{\eta_{1}}^{-1}\tau_{\eta_{2}}^{-1}\tau_{\eta_3}^{-1}\rhoG(g_{\eta_3}) \tau_{\eta_1}\tau_{\eta_0}$ \\
  \hline
  $\Sz(q)$ & $2^n$ & $\tau_{\eta_0}^{-1}\tau_{\eta_1}^{-1}\tau_{\eta_{2}}^{-1}\tau_{\eta_{3}}^{-1}\rhoG(g_{\eta_3})$  \\
  \hline
  \end{tabular}
  \caption{The map $Y_0\colon \MM_0\to \lieGroup$, for each of the groups $G$ in \cref{thmB}.}
  \label{TableY0}
 \end{table}
 \renewcommand{\arraystretch}{1}

 \begin{proposition}\label{propositionTNotSurjective}
  In all cases, the map $T\colon \prod_{i>0}\centralizer_\lieGroup( \rhoG( G_{\eta_i} ) )\to \lieGroup$ of \cref{TableZ} is not surjective.
  \begin{proof}
    Let $M=\prod_{i>0}\centralizer_\lieGroup( \rhoG( G_{\eta_i} ) )$ and let $H=\{\alpha\in \prod_{i\geq 1}\centralizer_\lieGroup( \rhoG( G_{v_i} ) ) \tq \alpha_{v_1}=\identityLieGroup\}$.
    The action $\MM_0\rightAction\HH$ restricts to a free action of $H\leq \HH$ on the factor $M$ of $\MM_0$.
    Note that $T$ factors through the quotient $M\to M/H$ giving a map $\overline{T}\colon M/H\to \lieGroup$.
    By \cref{QuotientManifoldTheorem}, $M/H$ is a manifold and we have
    \begin{align*}
     \dim M/H & =\dim M -\dim H \\
     & = \dim \MM_0 -\dim \HH - \dim \centralizer_{\lieGroup}(\rhoG(G_{\eta_0})) + \dim \centralizer_{\lieGroup}(\rhoG(G_{v_1})) \\
     & <\dim \MM_0 -\dim \HH\\
     & = \dim \lieGroup
    \end{align*}
    (by part (ii) of \cref{existenceSuitableRep}).
    Now by Sard's Theorem (\cite[Corollary 6.11]{Lee}) $\overline{T}$ is not surjective, and thus $T$ is not surjective.
  \end{proof}
  \end{proposition}

  \begin{table}[ht]
    \centering
     \begin{tabular}{|c|c|l|l|}
     \hline
     $G$ & $q$ & $T(\tau)$ & $Z(\tau)$  \\
     \hline
     $\PSL_2(q)$ & $2^n$  & $\tau_{\eta_1}^{-1}\tau_{\eta_2}^{-1}\rhoG(g_{\eta_2})$ & $\begin{array}{l}  A^{-1}\tau_{\eta_0}^{-1}A C \\ \cdot \tau_{\eta_1}^{-1}\tau_{\eta_2}^{-1}\rhoG(g_{\eta_2}) \end{array}$ \\
     \hline
     $\PSL_2(q)$ & $3^n$ & $\tau_{\eta_1}^{-1}\tau_{\eta_{2}}^{-1}\tau_{\eta_{3}}^{-1}\rhoG(g_{\eta_3})$ & $\begin{array}{l}  A^{-1}\tau_{\eta_0}^{-1}A C \\ \cdot \tau_{\eta_1}^{-1}\tau_{\eta_{2}}^{-1}\tau_{\eta_{3}}^{-1}\rhoG(g_{\eta_3})\end{array}$ \\
     \hline
     $\PSL_2(q)$ & $q\equiv 19\pmod{24}$ & $\tau_{\eta_1}^{-1}\tau_{\eta_{2}}^{-1}\tau_{\eta_{3}}^{-1}\rhoG(g_{\eta_3})$ & $\begin{array}{l}  A^{-1}\tau_{\eta_0}^{-1}A C \\ \cdot \tau_{\eta_1}^{-1}\tau_{\eta_{2}}^{-1} \tau_{\eta_{3}}^{-1}\rhoG(g_{\eta_3})\end{array}$ \\
     \hline
     $\PSL_2(q)$ & $q\equiv 11\pmod{24}$ & $\tau_{\eta_1}^{-1}\tau_{\eta_{2}}^{-1} \tau_{\eta_{3}}^{-1}\rhoG(g_{\eta_3}) \tau_{\eta_1}$ & $\begin{array}{l} 
     A^{-1}\tau_{\eta_0}^{-1}A C \\ \cdot  \tau_{\eta_1}^{-1}\tau_{\eta_{2}}^{-1} \tau_{\eta_{3}}^{-1}\rhoG(g_{\eta_3}) \tau_{\eta_1}\\
     \cdot   C^{-1} A^{-1} \tau_{\eta_0} A\end{array}$ \\
     \hline
     $\Sz(q)$ & $2^n$ & $\tau_{\eta_1}^{-1}\tau_{\eta_{2}}^{-1}\tau_{\eta_{3}}^{-1}\rhoG(g_{\eta_3})$ & $\begin{array}{l}A^{-1}\tau_{\eta_0}^{-1}A C \\ \cdot \tau_{\eta_1}^{-1}\tau_{\eta_{2}}^{-1}\tau_{\eta_{3}}^{-1}\rhoG(g_{\eta_3})\end{array}$ \\
     \hline
     \end{tabular}
     \caption{The definition of the maps $T$ and $Z$ for each of the groups $G$ in \cref{thmB}.}
     \label{TableZ}
    \end{table}

 \begin{proposition}\label{propositionZNotSurjective}
 In all cases, the map $Z\colon \MM_0\to \lieGroup$ of \cref{TableZ} is not surjective.
  \begin{proof}
    By \cref{propositionTNotSurjective}, it is enough to show the image of $Z$ is contained in the image of $T$.
    Let $\tau\in \MM_0$. It is straightforward to check that if $\widehat{\tau}\in \prod_{i>0}\centralizer_\lieGroup( \rhoG( G_{\eta_i} ) )$ is defined by
    \[\widehat{\tau}_{\eta_i} = \begin{cases}
      \tau_{\eta_1} \cdot C^{-1} A^{-1} \tau_{\eta_0} A & \text{ if } i=1 \\
      \tau_{\eta_i} & \text{ if } i\neq 1.
    \end{cases}\]
  then $Z(\tau)=T(\widehat{\tau})$.
  It follows from parts (iii) and (iv) of \cref{existenceSuitableRep} that $\widehat{\tau}$ is well-defined.
  \end{proof}
 \end{proposition}

\begin{proposition}\label{DegreeYY0bar}
 For each of the groups $G$ in \cref{thmB}, the degree of $\overline{Y}_0 \colon \MMbar_0 \to \lieGroup$ is $0$. 
\begin{proof}
By part (v) of \cref{existenceSuitableRep}, $A C$ commutes with $\centralizer_{\lieGroup}(\rhoG(G_{v_1}))$.
Moreover, by \cref{centralizersAreConnectedInUnitaryGroup}, the centralizers $\centralizer_\lieGroup(\centralizer_\lieGroup(\rhoG(G_{v_i})))$ are connected.
Therefore the hypotheses of \cref{inducedMapIsHomotopicToMapWithFactorsConjugated} are satisfied
and the map $Y_0\colon \MM_0\to \lieGroup$ is $\HH$-equivariantly homotopic to the map $Z$ defined in \cref{TableZ}.
Passing to the quotient we get an homotopy between the maps $\overline{Y}_0,\overline{Z}\colon \MMbar_0 \to \lieGroup$.
By \cref{propositionZNotSurjective}, $Z$ is not surjective and therefore $\overline{Z}$ is not surjective.
We conclude the degree of $\overline{Y}_0$ is $0$.
\end{proof}
\end{proposition}

\begin{corollary}\label{DegreeYYbar}
 The degree of $\YYbar=(\overline{Y}_0,\ldots,\overline{Y}_k)\colon \MMbar_k \to \lieGroup^{k+1}$ is $0$.
\begin{proof}
We have $\MMbar_k = \MMbar_0 \times \lieGroup^k$. 
Now, by \cref{DegreeYY0bar}, the map $\YYbar\colon \MMbar_0\times \lieGroup^k\to \lieGroup^{k+1}$ has degree $0$
since it is the product of the map $\overline{Y}_0\colon \MMbar_0\to \lieGroup$
and the maps $\overline{Y}_i\colon \lieGroup\to \lieGroup$
given by $\tau_{\eta'_i}\mapsto \tau^{-1}_{\eta'_i}$ for $i=1,\ldots,k$.
\end{proof}

 \end{corollary}

\begin{proposition}\label{DegreeComputation}
 Let $w_0,\ldots, w_k\in \Gamma_k$ and let $\WWbar=(\overline{W}_0,\ldots,\overline{W}_k)\colon \MMbar_k \to \lieGroup^{k+1}$.
 Then $\deg(\WW )=0\in\Z$.
\begin{proof}
  First note that, by \cref{lemmaRemovingConstantsGeneralization} (and \cref{centralizersAreConnectedInUnitaryGroup}), we only need to address the case when the $w_i$ are words in the generators $y_0,\ldots, y_k$.
  Now consider the map $\YYbar=(\overline{Y}_0,\ldots,\overline{Y}_k)$ and consider the map $\wt{\WW}\colon \lieGroup^{k+1}\to \lieGroup^{k+1}$ induced by the words $w_0,\ldots,w_k\in F(y_0,\ldots, y_k)$, which makes the following diagram commute
  \begin{center}
  \begin{tikzcd}
   \MMbar_k \arrow{r}{\YYbar} \arrow{dr}[swap]{\WWbar} & \lieGroup^{k+1}\arrow{d}{\wt{\WW}} \\
   & \lieGroup^{k+1}
  \end{tikzcd}
\end{center}
By  \cref{DegreeYYbar} $\YYbar$ has degree $0$ and since $\deg(\WWbar)=\deg(\wt{\WW})\cdot\deg(\YYbar)$ we are done.
\end{proof}
\end{proposition}


\section{Group actions on contractible $2$-complexes}

We are now ready to prove the main results of this article.

\begin{theorem}\label{goodRepresentation}
 Let $G$ be one of the groups in \cref{thmB}.
 Let $w_0,\ldots, w_k\in N$. If $N=\llangle w_0,\ldots, w_k\rrangle^{\Gamma_k}[N,N]$ then there is a point $\taubar\in \MMbar_k$ such that
 
 (i) $\rhobar_\taubar(w_i)=1$ for $i=0,\ldots, k$, and
 
 (ii) $\rhobar_\taubar$ is not universal.
 \begin{proof}
 By \cref{DegreeComputation} the degree of $\WWbar$ is $0$.
 By \cref{universalRepIsRegularPoint}, $\identityMMBar$ is a regular point of $\WWbar$.
 Therefore, there must exist a point $\taubar\in \WWbar^{-1}(\identityLieGroup)$ with $\taubar\neq\identityMMBar$.
 To conclude note that by \cref{universalIsUnique}, $\taubar$ is not universal.
 \end{proof}
\end{theorem}

\begin{proof}[Proof of \cref{thmB}]
  By \cref{Refinement} it is enough to prove the result when $X$ is obtained from $\XOSk(G)$ by attaching $k+1$ free orbits of $2$-cells.
  By \cref{Brown}, there are words $w_0,\ldots,w_k\in N$ such that $\pi_1(X)\groupIso \frac{N}{\llangle w_0,\ldots, w_k\rrangle^{\Gamma_k}}$ and since $H_1(X)=0$ we have $N=\llangle w_0,\ldots,w_k\rrangle^{\Gamma_k}[N,N]$.
  Now passing to the quotient the representation $\rhobar_{\taubar}$ given by \cref{goodRepresentation} we obtain a nontrivial representation $\pi_1(X)\to \U(m)$.
 \end{proof}

 Recall the following basic result from the theory of equations over groups.

 \begin{proposition}[{\cite[Proposition 2.3 (i)]{HowieEquations}}]\label{propSystemEquations}
   Let $X$ be a finite acyclic $2$-complex and let $A\subset X$ be an acyclic subcomplex.
   Then we can write
   \[\pi_1(X)=(\pi_1(A)*F(x_1,\ldots,x_n))/\llangle w_1,\ldots, w_n\rrangle\]
   and the $(n\times n)$-matrix $M$ such that $M_{i,j}$ is the total exponent of $x_j$ in $w_i$ is invertible.
 \end{proposition}
 
 The Gerstenhaber--Rothaus theorem~\cite{GerstenhaberRothaus} has the following immediate consequence.
 
 \begin{proposition}\label{rephrasingeOfGerstenhaberRothaus}
   Let $X$ be a finite acyclic $2$-complex, $A\subseteq X$ an acyclic subcomplex and  $\rho\colon \pi_1(A)\to \lieGroup$ a nontrivial representation into a compact and connected Lie group $\lieGroup$.
   Then there is a nontrivial representation $\overline{\rho}\colon \pi_1(X)\to \lieGroup$ such that $\overline{\rho}\circ \pi_1(i)=\rho$, where $i\colon A\to X$ denotes the inclusion.
   \begin{proof}
       We write $\pi_1(X)=\pi_1(A) * F(x_1,\ldots, x_n) / \llangle w_1,\ldots, w_n\rrangle$ using \cref{propSystemEquations}.
       There is an induced map
       $\rho \colon \pi_1(A) * F(x_1,\ldots, x_n) \to \lieGroup * F(x_1,\ldots, x_n)$.
       Now~\cite[Theorem 1]{GerstenhaberRothaus} gives elements $x_1,\ldots, x_n\in \lieGroup$ satisfying the equations $\rho(w_1),\ldots, \rho(w_n)$ and the desired representation is obtained by passing to the quotient.
   \end{proof}
 \end{proposition} 

\begin{proof}[{Proof of \cref{CasacubertaDicks}}]
  Let $G$ be a finite group and suppose that $X$ is a finite, acyclic $2$-dimensional fixed point free $G$-complex.
  Let $N$ be the subgroup generated by all normal subgroups $N'\triangleleft G$ such that $X^{N'}\neq \emptyset$.
  By \cref{OSThmB} we have that $Y=X^N$ is acyclic and the action of $K=G/N$ on $Y$ is essential and fixed point free.
  Then $K$ must be one of the groups in \cref{OSThmA}.
  Since \cref{thmB,thmC} together cover all the groups in \cref{OSThmA},
  it follows that $\pi_1(Y)$ admits a nontrivial unitary representation.
  Therefore, by \cref{rephrasingeOfGerstenhaberRothaus}, $\pi_1(X)$ also admits a nontrivial unitary representation.
\end{proof}


\bibliographystyle{alpha}
\bibliography{references}

\end{document}